\documentclass[a4paper,11pt]{amsart}
\usepackage{amsmath,amsthm,amssymb,amsfonts}
\usepackage{caption}
\allowdisplaybreaks
\usepackage{enumitem}
\usepackage{soul}

\usepackage[colorlinks=true, allcolors=blue]{hyperref}

\usepackage[usenames, dvipsnames]{color}
\usepackage[ansinew]{inputenc}
\usepackage[english]{babel}
\usepackage{tikz}
\usepackage{url}

\providecommand{\U}[1]{\protect\rule{.1in}{.1in}}
\newtheorem{theorem}{Theorem}

\newtheorem{condition}[theorem]{Condition}

\newtheorem{corollary}[theorem]{Corollary}

\newtheorem{definition}[theorem]{Definition}

\newtheorem{lemma}[theorem]{Lemma}

\newtheorem{proposition}[theorem]{Proposition}
\newtheorem{remark}[theorem]{Remark}

\usepackage[numbers,sort]{natbib}

\newcommand{\red}[1]{{#1}}

\newcommand{\B}{\mathbb{B}}

\newcommand{\E}{\mathbb{E}}

\newcommand{\N}{\mathbb{N}}

	\renewcommand{\P}{\mathbb{P}}

\newcommand{\R}{\mathbb{R}}
\newcommand{\T}{\mathbb{T}}

\newcommand{\Z}{\mathbb{Z}}

\newcommand{\wt}{\widetilde}

\newcommand{\ep}{\epsilon}

\newcommand{\cN}{\mathcal{N}}

\newcommand{\cX}{\mathcal{X}}

{}

{}

{}

\numberwithin{equation}{section}
\numberwithin{theorem}{section}

\theoremstyle{definition}

\begin{document}

\title[Scaling Limit of the Kuramoto Model on Random Geometric Graphs]{Scaling Limit of the Kuramoto Model  on Random Geometric Graphs}

\author[F. Cirelli]{Francisco Cirelli}
\address{Departamento de Matem\'atica\hfill\break \indent Facultad de Ciencias Exactas y Naturales\hfill\break \indent Universidad de Buenos Aires\hfill\break \indent Buenos Aires, Argentina}
\email{franciscocirelli@gmail.com}

\author[P. Groisman]{Pablo Groisman}
\address{Departamento de Matem\'atica\hfill\break \indent Facultad de Ciencias Exactas y Naturales\hfill\break \indent Universidad de Buenos Aires\hfill\break \indent IMAS-UBA-CONICET\hfill\break \indent Buenos Aires, Argentina}
\email{pgroisma@dm.uba.ar}

\author[R. Huang]{Ruojun Huang}
\address{Fachbereich Mathematik und Informatik\hfill\break \indent Universit\"at M\"unster \hfill\break \indent Einsteinstr. 62, M\"unster 48149, Germany.}
\email{ruojun.huang@uni-muenster.de}

\author[H. Vivas]{Hern\'an Vivas}
\address{Centro Marplatense de Investigaciones Matem\'aticas/Conicet}
\address{De\'an Funes 3350, 7600, Mar del Plata, Argentina}
\email{havivas@mdp.edu.ar}


\keywords{interacting dynamical systems; scaling limit; Kuramoto model; random geometric graphs}
\subjclass{Primary 34C15; Secondary 05C80, 34D06}

\begin{abstract}
We consider the Kuramoto model on a graph with nodes given by $n$ i.i.d.~points uniformly distributed on the $d$ dimensional torus. Two nodes are declared neighbors if they are at distance less than $\epsilon$. We prove a scaling limit for this model in compact time intervals as $n\to\infty$ and $\epsilon \to 0$ such that $\epsilon^{d+2}n/\log n \to \infty$. The limiting object is given by the heat equation. On the one hand this shows that the nonlinearity given by the sine function disappears under this scaling and on the other hand, provides evidence that stable equilibria of the Kuramoto model on these graphs are, as $n\to\infty$, in correspondence with those of the heat equation, which are explicit and given by twisted states. In view of this, we conjecture the existence of twisted stable equilibria with high probability as $n\to \infty$.
\end{abstract}
\date{\today}
\maketitle


\section{Introduction}\label{sec:intro}

Phase synchronization of systems of coupled oscillators is a phenomenon that has attracted the mathematical and scientific community for centuries, both because of its intrinsic mathematical interest \cite{chiba2016mean, medvedev2018continuum, bertini2014synchronization, coppini2020law} and since it appears in a wide range of physical and biological models \cite{mirollo1990synchronization, winfree1967biological, acebron2005kuramoto, bullo2020lectures, arenas2008synchronization, dorflerSurvey, strogatz2004sync, strogatz2000kuramoto}.

One of the most popular models for describing synchronization of a system of coupled oscillators is the Kuramoto model. Given a finite graph $G=(V,E)$ with $|V|=n$, the Kuramoto model determined by $G$ is the following ODE system 
\begin{equation}
\label{kuramoto.graph}
\frac{d}{dt}\theta_i(t)=\omega_i + \sum_{j\in V} w_{ij}\sin(\theta_j(t)-\theta_i(t)), \quad i \in V.
\end{equation}
Here $\theta_i\in[0,2\pi)$ represents the phase of the $i$-th oscillator, $\omega_i$ its natural frequency and the nonnegative weights $w_{ij}$, that verify $w_{ij}>0 \Leftrightarrow \{i,j\}\in E$, account for the strength of the coupling between two connected oscillators. Our focus in this article is on homogeneous Kuramoto models with $\omega_i=0$ for all $i$, related to phase synchronization.

With origins in the study of chemical reactions and the behavior of biological populations with oscillatory features \cite{kuramoto1975self,kuramoto1984chemical}, the Kuramoto model has proved to be applicable in the description of phenomena in areas as varied as neuroscience \cite{cumin2007generalising,breakspear2010generative}; superconductors theory \cite{wiesenfeld1996synchronization}; the beating rhythm of pacemaker cells in human hearts \cite{peskin1975mathematical} and the spontaneous flashing of populations of fireflies \cite{mirollo1990synchronization}. The reader is referred to  the surveys  \cite{dorfler2013synchronization,acebron2005kuramoto,rodrigues2016kuramoto, strogatz2000kuramoto} and the references therein for a more complete picture of the advances on the topic.

The model has been studied both by means of rigorous mathematical proofs and heuristic arguments and simulations in different families of graphs.

\subsection{Some background on scaling limits for the Kuramoto model}

A particularly interesting problem is to understand the behavior of the system as the size of the graph goes to infinity, usually referred as the {\em scaling limit}. This has been carried out for graphons \cite{medvedev2018continuum, Medvedev2014, MedvedevWgraphs, MedvedevSmallWorld}, Erd\H{os}-R\'enyi graphs \cite{medvedev2018continuum, ling2019landscape, kassabov2022global}, small-world and power-law graphs \cite{MedvedevSmallWorld, medvedev2018continuum}.

The Kuramoto model in random geometric graphs has been studied in \cite{abdalla2022guarantees,devita2024energy}. In \cite{abdalla2022guarantees} the authors are interested in the optimization landscape (i.e.~the study of local minima) of the energy function determined by \eqref{eq:km} rather than the scaling limit of the solution. They also work on a different regime: in their setting the graphs are constructed on the sphere $\mathbb S^{d-1}$ rather than in the torus and $d\to \infty$ as $n\to \infty$. In that context, they obtain guarantees for global spontaneous synchronization (i.e.~the global minimum $\theta_1=\theta_2 = \cdots = \theta_n$ is the unique local minimum of the energy). This is pretty different from our situation as we will see. In \cite{devita2024energy} the setting is similar to ours but restricted to dimension one. In that case the authors prove the existence of twisted states of arbitrary order as $n\to \infty$ with high probability.

From a different perspective, the scaling limit of the empirical measure (i.e.~considering the proportion of oscillators at each state instead of the state of each oscillator) has been largely studied, from the seminal phenomenological work of Ott and Antonsen \cite{OA1,OA2} to the rigorous mathematical proofs \cite{coppini2020law, coppini2022long, Oliveira1, Oliveira2, Oliveira3} among others.

A different type of scaling limit has been studied in \cite{GHV}. In that work the size of the graph is fixed but the connections are random and time-dependent (i.e., edges appear and disappear in a random way as time evolves). The authors obtain a deterministic behavior as the rate of change of the connections goes to infinity. This kind of scaling is usually called {\em averaging principle}.

The goal of this work is to study the scaling limit of \eqref{kuramoto.graph} in random geometric graphs as the size of the graph goes to infinity, which is not contained in all the previously mentioned works.

In our setting, the nodes of the graphs are contained in Euclidean space, and the neighboring structure is given by geometric considerations in such a way that the dimension of the space and the distribution of the points are crucial to determine the scaling limit of the model.

There is an extensive body of literature on continuum limits for PDEs and variational problems on random geometric graphs. We do not intend to mention all the work since it is huge. In the last ten years, there has been enormous progress in proving such limits both qualitatively and quantitatively, both for linear and nonlinear cases, both in Euclidean spaces and on manifolds. The convergence results are impressive. We mention just a few of them  \cite{thorpe2015gamma, trillos2016continuum, calder2019game, trillos2020error, armstrong2023quantitative, bungert2024ratio}, and refer the reader to the references therein for more. \red{Many of these works do not consider time-dependent equations, but we also note that parabolic problems have been studied in recent years, while the limit equations obtained there can be more complicated than ours, e.g.~evolutionary $p$-Laplace equation, see for example \cite{yuan2022continuum, esposito2021nonlocal, weihs2024discrete} and references therein. As stated before, we do not aim to cover all the bibliography here and we apologize for possible omissions that would certainly occur.}

\subsection{Proposed model and main results}\label{subsec:setting}

Given any positive integer $n$ and $\ep=\ep(n)\in(0,1)$, let $V=\{x_1,...,x_n\}$ be a set of $n$ independent and identically distributed (i.i.d.) uniform points in \red{$\Lambda_d:=[0,2\pi]^d\subset\R^d$. We also denote $\T^d=(\R/2\pi\Z)^d$, the $d$-dimensional flat-torus of side length $2\pi$ with periodic boundaries}. When $d=1$ we write $\T^1 = \mathbb T$. 

\red{Let $q:\mathbb{R}^d \to \mathbb{T}^d$ be the quotient map to the torus and $d_{\mathbb{T}^d}: \mathbb{T}^d \times \mathbb{T}^d \to \mathbb{R}_{\geq 0}$ be the usual flat-torus distance, i.e,
\[
d_{\T^d}(q(x_1,\ldots,x_d),q(y_1,\ldots,y_d)) = \sqrt{\sum_{i=1}^d \|x_i - y_i\|^2_{\mathbb{R}/2\pi\mathbb{Z}}}
\]
where $\|\cdot\|_{\mathbb{R}/2\pi\mathbb{Z}}$ is the distance to the closest integer multiple of $2\pi$.

\begin{remark}
\label{rmk:tor-dist}
For technical reasons, we will extend the point cloud $V$ on the box $\Lambda_d$ to a point cloud $\tilde V$ on the whole space $\R^d$ periodically. To do that, we first copy the points $V$ from $\Lambda_d=[0,2\pi]^d$ to every box of the form $\Pi_{\ell=1}^d[2\pi m_\ell, 2\pi (m_\ell+1)]$, for $m_\ell\in\Z$, and these boxes form a tiling of $\R^d$. Each point $x_i\in V$ in the original cloud has copies of the form $x_i+\sum_{\ell=1}^d2\pi m_\ell \vec e_\ell$, where $\{\vec e_\ell\}_{\ell=1}^d$ is the canonical basis of $\R^d$. Clearly, if $\tilde x_i$ is a copy of $x_i$, $q(\tilde x_i)=q(x_i)$.

Note that for each $x_i,x_j \in V$, there is a copy of $x_j$ in $\tilde V$, which we will call $x^i_j$, which is closest to $x_i$ among all other copies of $x_j$. An important fact to note is that $\|x_i - x^i_j\|_{\mathbb{R}^d} = d_{\mathbb{T}^d} (q(x_i),q(x_j))$.
\end{remark}
} 

From now on, we denote $|\cdot|=\|\cdot\|_{\R^d}$. We denote with $\sigma_d$ the volume of a unit ball in $\R^d$, and let $K:\R^d\to\R_{\ge0}$ be a given non-negative, {\it{radially symmetric}} function with compact support in the closure of the unit ball $\B(0,1)$ such that $K(z)>0$ for every $z\in\B(0,1)$ and $$\int_{\R^d} K(z)dz=1.$$ 
We denote the finite constants
\begin{align}\label{kappa-2}
\kappa_i:= \frac{1}{\sigma_d}\int_{\R^d}|z|^iK(z)\, dz, \quad i=1,2.
\end{align}


\red{We consider the homogeneous Kuramoto model on a weighted graph $G=(V,E)$, where $V=\{x_1,...,x_n\}$ and $\{x_i,x_j\}\in E \Leftrightarrow d_{\mathbb{T}^d}(q(x_i),q(x_j) = |x_i-x^i_j|<\epsilon$. Thus, two points are connected if and only if they are close when seen as points on the torus. The weights are given by $w_{ij}=K(\ep^{-1}(x^i_j-x_i))$. Note now, since $K$ is radially symmetric, the weight is solely determined by the distance of the points in the torus. Let $u^{n}\colon [0,\infty)\times V\to\R$ be the unique solution to a system of $n$ 
 Kuramoto equations 
\begin{align}
\begin{cases}
\displaystyle{\frac{d}{dt}}u^{n}(t,x_i)=  \displaystyle{\frac{1}{\ep^2N_i}}\sum_{j=1}^n\sin \left(u^{n}(t,x_j)-u^{n}(t,x_i)\right)K\left(\ep^{-1}(x^i_j-x_i)\right), \label{eq:km}\\[10pt]
u^{n}(0,x_i)=u^n_0(x_i), \quad\quad  i=1,2,...,n.
\end{cases}
\end{align}
The random integer $N_i=|\cN(i)|$ denotes the cardinality of the set of neighbors of $x_i$, that we call
$$
\cN(i) :=  \{j \colon j\neq i,\, d_{\mathbb{T}^d}(q(x_i),q(x_j))=|x^i_j-x_i| < \ep\} \subset\{1,2,...,n\}\backslash\{i\}.
$$
}
Our assumptions on the kernel $K$ include the canonical choice of the indicator function on the unit ball (normalized to have integral $1$), in which case we get constant weight for points at distance smaller than $\epsilon$ and zero otherwise. \red{The Kuramoto model can be also interpreted as taking values in the unit circle $\mathbb S^1\cong\T$, where each $u^n(t,x_i)$ represents the phase of an oscillator, and this is done by  performing $u^n$ mod $2\pi$.}

\red{
\begin{remark}\label{rmk:km-extension}
    We extend the initial condition $u_0^n$ from $V$ to $\tilde V$ in the following ``pseudo-periodic" way: we will choose some $\{\tilde k_\ell\}_{\ell=1}^d\in\Z^d$ such that if $y\in\tilde V$ satisfies $y-x_i = \sum_{\ell=1}^d 2\pi m_\ell\vec e_\ell$ for some $x_i\in V$ and $\{m_\ell\}_{\ell=1}^d\in\Z^d$, then we set 
    \[
    u_0^n(y)= u_0^n(x_i)+\sum_{\ell=1}^d2\pi \tilde k_\ell m_\ell.
    \].

  Similarly, we extend $u^n$ to $\tilde V$ via
    
    \begin{align}\label{extended-ode}
    u^n(t,y):= u^n(t,x_i)+\sum_{\ell=1}^d2\pi \tilde k_\ell m_\ell, \quad \text{for all }t\ge0.
    \end{align}
    Note that, with this extension, for $x_i \in V$, 
    \begin{align}\label{eq:ext-kur}
        \sum_{j=1}^n\sin \left(u^{n}(t,x_j)-u^{n}(t,x_i)\right)K\left(\ep^{-1}(x^i_j-x_i)\right) &= \sum_{j=1}^n\sin \left(u^{n}(t,x^i_j)-u^{n}(t,x_i)\right)K\left(\ep^{-1}(x^i_j-x_i)\right)\nonumber\\ 
        &= \sum_{y \in \tilde V}\sin \left(u^{n}(t,y)-u^{n}(t,x_i)\right)K\left(\ep^{-1}(y-x_i)\right)
    \end{align}
    where the last equality stems from the fact that $K$ is supported in $\mathbb{B}(0,1)$ and at most one copy of each $x_j$ is within $\ep$-distance of $x_i$. This shows that our extension to $u^n$ can be thought of as a solution to the Kuramoto model on the point cloud $\tilde V$, where points $x,y$ are connected if they are $\ep$-close with weight $K(\ep^{-1} (x_i - x_j))$.

\end{remark}
Throughout the paper we use the notation
\begin{align}\label{def:mod}
[y]:=y \text{ mod }2\pi \quad\in\Lambda_d,
\end{align}
the modulo being performed component-wise, for any $y\in\R^d$. }

For any fixed $n$ and realization of the random points $V=\{x_1,...,x_n\}$, \eqref{eq:km} is a finite system of ODEs $(u^1,...,u^n)$ with Lipschitz coefficients, hence existence and uniqueness of solution is classical.

To build intuition on the scaling limit behavior of \eqref{eq:km} as $\ep\to0$ and $n\to\infty$, it is helpful to notice that two neighboring points $x_i,x_j$ in our geometric graph are at distance at most $\ep$, as enforced by the kernel $K$, hence if $u^n$ is suitably regular, the argument of the sine function is also very close to $0$. By Taylor expansion at $0$, $\sin t \approx t$, thus the nonlinear operator in \eqref{eq:km} is close to the discrete Laplacian. In fact,
the sine function in \eqref{eq:km} can be replaced by an odd $2\pi$-periodic smooth function $J$ with Taylor expansion $J(x)= J'(0)x + o(x^2)$, with $J(0)=0, J'(0) > 0$. In that case it can be assumed without loss of generality that $J' (0)=1$.

We are interested in the Kuramoto model in graphs with this structure, on the one hand since they are ubiquitous when modeling interacting oscillators with spatial structure, and on the other hand because they form a large family of model networks with persistent behavior (robust to small perturbations) for which we expect to have twisted states as stable equilibria. For us, a twisted state is an equilibrium solution of \eqref{eq:km} for which the vector $(\tilde k_1, \dots, \tilde k_d)$ defined in \eqref{extended-ode} below is not null.

We remark that spatial structure and local interactions have been shown to be crucial for the emergence of patterns in this kind of synchronized systems for chemical reactions \cite{winfree1984organizing}, behavior of pacemaker cells in human hearts \cite{peskin1975mathematical} and the spontaneous flashing of populations of fireflies \cite{mirollo1990synchronization}.

Twisted states have been identified in particular classes of graphs as explicit particular equilibria of \eqref{eq:km}. They have been shown to be stable equilibria in rings in which each node is connected to its $k$ nearest neighbors on each side \cite{wiley2006size}, in Cayley graphs and in random graphs with a particular structure \cite{MedvedevStability}. In all these cases, graph symmetries (which are lacking in our model) are exploited to obtain the twisted states. Patterns in the the Kuramoto model are also reported in \cite{chiba2022instability}.

For $k \in \Z$ and $1 \le \ell \le d$ we call $\vec e_\ell$ the $\ell$-th canonical vector and consider the functions $u_{k,\ell}\colon \T^d \to \T$ given by,
\begin{align}\label{cont-twist}
u_{k,\ell}(x)=kx\cdot \vec e_\ell  \quad \text{mod } 2\pi.
\end{align}

We are thinking of twisted states as stable equilibria that are close in some sense to (or that can be put in correspondence with) the functions $u_{k,\ell}$. \red{We remark that the functions $u_{k,\ell}$ are defined mod $2\pi$. As a consequence, these functions are stable equilibrium solutions of the heat equation} \red{for $u:\T^d\to\T$}
\begin{equation}
\label{eq:heat}
\left\{
\begin{array}{rcl}
\displaystyle{\frac{d}{dt}} u(t,x) & = & \displaystyle{\frac{\kappa_2}{2d}}\Delta u(t,x),\\[10pt]
u(0,x) & = & u_0(x),
\end{array}
\right.
\end{equation}
see \cite{wiley2006size, MedvedevContTwist, MedvedevStability, MedvedevSmallWorld}\footnote{Even if these references do not deal with the heat equation but rather with different versions of the Kuramoto model, the stability analysis for \eqref{eq:heat} is similar and even simpler since it is a linear equation and the Fourier series can be explicitly computed.}. We will call them {\em continuous twisted states}; these are all the equilibiria of \eqref{eq:heat}. \red{Since PDEs taking values in $\T$ are not standard, we need to correctly interpret such an equation. In this paper, by saying $u:[0,+\infty)\times\T^d\to\T$ solves the heat equation \eqref{eq:heat} we mean its lift as a function from $\R^d$ to $\R$ solves the usual heat equation, whose initial condition $\tilde u_0$ is the lift of $u_0$. This is done in the definition and lemma below.}


\red{
\begin{definition}\label{def:pseudo}
    We say that a function $\tilde f:\R^d\to\R$ is {\it{pseudo-periodic}} (of period $2\pi$), if there exist $k_\ell\in\Z$, $\ell=1,2,...,d$, such that 
    \begin{align}\label{eq:pseudo-periodic}
    \tilde f(x+2\pi \vec e_\ell)= \tilde f(x)+ 2\pi k_\ell
    \end{align}
    holds for any $x\in\R^d$. We call $k_\ell\in\Z$, $\ell=1,2,...,d$ the winding numbers of $\tilde f$.
\end{definition}

\begin{lemma}\label{lem:winding}
    To any continuous function $f:\T^d\to \mathbb T$ corresponds a pseudo-periodic continuous function $\tilde f:\R^d\to\R$, whose winding numbers are determined by $f$ and is  unique up to a global shift of an integer multiple of $2\pi$. We call $\tilde f$ the lift of $f$. If $f:[0,\infty)\times\T^d\to\mathbb T$ further depends on time and is continuous in $(t,x)$, then it has a continuous-in-$(t,x)$ lift $\tilde f(t,x)$ whose winding numbers are constant in $t$ and determined by $f|_{t=0}$. Conversely, for any pseudo-periodic continuous $\tilde f:\mathbb{R}^d \to \mathbb{R}$ corresponds a continuous $f:\mathbb{T}^d \to \mathbb{T}$. Similarly, for any pseudo-periodic in space, continuous $\tilde f:[0,\infty) \times \mathbb{R}^d \to \mathbb{R}$ corresponds a continuous $f:[0,\infty) \times \mathbb{T}^d \to \mathbb{T}$. 
\end{lemma}
\begin{proof}

Let $p:\mathbb{R} \to \mathbb{T} = \mathbb S^1$ be the standard universal cover of $\mathbb S^1$, i.e, $p(x) = e^{\mathbf ix}$ where $\mathbf i$ is the complex unit. Given a continuous $f:\mathbb{T}^d \to \mathbb{T}$, $f\circ q:\mathbb{R}^d \to \mathbb{T}$. Therefore, since $\mathbb{R}^d$ is simply connected, given $e_0 \in p^{-1}(f \circ q(0)) = \{e_0  + 2k\pi : k \in \mathbb{Z}\}$, there exists a unique continuous $\tilde f:\mathbb{R}^d \to \mathbb{R}$ such that $p \circ \tilde f = f \circ q$ and $\tilde f(0) = e_0$. 

Since $p(x) = p(y)$ if and only if $x-y$ is an integer multiple of $2\pi$ and $p\circ \tilde f(x+2\pi \vec e_\ell) = f \circ q(x+2\pi \vec e_\ell) = f \circ q(x) =p \circ \tilde f(x)$, for each $x$ there exist $k_\ell$ such that $\tilde f(x + 2\pi\vec e_\ell) = \tilde f(x) + 2\pi k_\ell$. However, since $\tilde f$ is continuous, these integers must be constant for all $x$.

Given any other $\tilde g:\mathbb{R}^d \to \mathbb{R}$ such that $p \circ \tilde g = f \circ q$, note that $\tilde g(0) \in p^{-1}(f \circ q(0))$ and so $\tilde f(0) = \tilde g(0) + 2k\pi$ for some integer $k$. Therefore, since $p \circ \tilde f = p \circ \tilde g = p \circ (\tilde g + 2k\pi)$, by the uniqueness of the lift, $\tilde f = \tilde g + 2k\pi$, showing that they differ by an integer multiple of $2\pi$. Moreover, noting that the winding numbers $k_{\ell} = \frac{\tilde f(2\pi \vec e_\ell) - \tilde f(0)}{2\pi} = \frac{\tilde g(2\pi \vec e_\ell) - \tilde g(0)}{2\pi}$, we see that the winding numbers depend solely on $f$.

For the converse result, given a continuous pseudo-periodic $\tilde f:\mathbb{R}^d \to \mathbb{R}$, note that $f :\mathbb{T}^d \to \mathbb{T}$ given by $f(q(x)) := p(f(x))$ is well defined and since $q$ is a quotient map, $f$ is continuous.

Utilizing the fact that $\red{1_{[0,\infty)} \times p:[0,\infty) \times \mathbb{R}^d \to [0,\infty) \times \mathbb{T}^d}$ is a universal covering of $[0,\infty) \times \mathbb{T}^d$ as well as the fact that the winding numbers are time independent due to continuity and analogous reasoning, we obtain the equivalent results for the correspondence in $[0,\infty) \times \mathbb{T}^d$.
\end{proof}



This correspondence allows us to interpret functions $f:\mathbb{T}^d \to \mathbb{T}$ and pseudo-periodic $\tilde f:\mathbb{R}^d \to \mathbb{R}$ as one and the same. In particular, the lift of the continuous twisted states \eqref{cont-twist} are just the linear functions $kx\cdot \vec e_\ell$. In view of Lemma \ref{lem:winding}, when we speak of solutions to \eqref{eq:heat} we mean their lifts. }

\begin{remark}\label{rmk:extension}

Both the solutions of \eqref{eq:km} and \eqref{eq:heat} are shift-invariant. Hence, we can assume without loss of generality (and we will do so) that the average of the initial condition is zero (i.e. $\sum_{i=1}^n u_0(x_i)=0$ in \eqref{eq:km} and $\int_{\Lambda_d} \tilde u_0(x) \, dx = 0$ in \eqref{eq:heat}). This is preserved for all times. When referring to the dynamics of both equations (basins of attraction, asymptotical stability etc.), we are implicitly assuming that the phase space for the dynamics is given by this restriction: the orthogonal space of $(1, \dots, 1)$ for \eqref{eq:km} and functions with zero-average --- orthogonal to the constants --- for \eqref{eq:heat}. \red{This is standard when working with the Kuramoto model and we refer the reader to \cite[Chapter 17]{bullo2020lectures} for more details.}
\end{remark}

Twisted states are expected to be robust and persistent in several contexts (since they can be observed in nature), but situations in which they can be computed explicitly (and hence proving their existence) are not. In our model, we expect to have twisted states as a generic property, i.e.~we expect them to exist with high probability as $n\to\infty$ and to persist if one adds or removes one or a finite number of points; or if one applies small perturbations to the points in a generic way.

Although the existence of such steady-states cannot be deduced directly from our arguments, we think that our results provide evidence of their ubiquity as a robust phenomenon.

Twisted states and their stability have also been studied in small-world networks \cite{MedvedevSmallWorld} and in the continuum limit \cite{MedvedevContTwist} among others.

We will work under the following assumption.
\begin{condition}\label{assump-eps}
$\ep\to0$ as $n\to\infty$ and 
\[
\liminf_{n\to\infty}\frac{\ep^{d+2}n}{\log n}=\infty.
\]
\end{condition}
This is contained in what is usually called the {\em sparse} regime in the synchronization community\footnote{It is curious that from the point of view of synchronization, this regime can be thought of as sparse, since the degree of each node is of a smaller order than the number of nodes in the graph. However, from the point of view of random geometric graphs, this regime is not {\em sparse} but the opposite, since we are in a supercritical regime from the point of view of connectivity.}.  It is worth to note that Condition \ref{assump-eps} is the threshold for the pointwise convergence of the graph Laplacian (see the discussion in the introduction of \cite{trillos2020error} and \cite[Section 2.2]{armstrong2023quantitative}). We think this is not the optimal rate to obtain our results. We expect them to hold up to the rate
\begin{align*}
\liminf_{n\to\infty}\frac{\ep^{d}n}{\log n}=\infty,
\end{align*}
which is (up to logarithmic factors) the connectivity threshold and also guarantees that the degree of each node goes to infinity faster than $\log n$ \cite{penrose1995single}.

If $\ep \to 0$ at an even smaller rate, different behaviors are expected depending on the rate of convergence. It is a very interesting problem to obtain such behaviors. This kind of situations have previously appeared in the aforementioned references on continuum limit of variational problems in random geometric graphs. See for example \cite{trillos2020error}.

Note that since $n\epsilon^d \to \infty$, we have
\[
\P\left(N_i=\frac{\sigma_d\ep^dn}{(2\pi)^d}\left(1+o(1)\right),\,  i=1,2,...,n\right)\to 1.
\]
In fact, by Bernstein's inequality and union bound, we have that 
\begin{align}\label{bernstein}
\P\left(\sup_{x_i\in V}\left|N_i-\frac{\sigma_d\ep^dn}{(2\pi)^d}\right|>\lambda\right)\le 2n \exp\left\{-\frac{2\lambda^2}{\frac{\sigma_d \ep^d n}{(2\pi)^d}+\frac{\lambda}{3}}\right\},\quad \lambda>0,
\end{align}
which implies the previous statement.

By means of this identification, we can compare solutions of \eqref{eq:km} with solutions of \eqref{eq:heat} and that is the purpose of our main theorem below. Denote $\|u^n_0 - u_0\|_{L^\infty(V)} = \sup_{x_i\in V}|u^n_0(x_i) - u_0(x_i)|$.

\begin{theorem}\label{thm.main}
Let $T>0$ be fixed, $u^n: [0,T]\times V\to\R$ be the unique solution of \eqref{eq:km} with initial condition $u_0^n:V\to\R$, and $u:[0,T]\times\R^d\to\R$ the unique \red{pseudo-periodic} solution of \eqref{eq:heat} with \red{pseudo-periodic} initial condition $\tilde u_0\in C^{2,\alpha}(\R^d, \R)$ for some $\alpha\in(0,1)$. Assume Condition \ref{assump-eps} holds and $$\sum_{n=1}^ \infty  \P\left(\|u^n_0 - u_0\|_{L^\infty(V)} > \delta \right)<\infty$$ for every $\delta>0$. Then,
\begin{equation*}
\lim_{n\to \infty }\sup_{x_i\in V,\, t\in[0,T]} |u^n(t,x_i) - u(t,x_i)| = 0, \quad \text{almost surely.}
\end{equation*}
\end{theorem}

\begin{proof}
The proof is a consequence of Proposition \ref{prop.convergence} and Proposition \ref{ppn:kur-to-int} below.
\end{proof}

\begin{remark}
We can give a quantitative bound on the difference between Kuramoto solution and the heat equation. This is because most parts of our proof are non-asymptotic. Fixing $T$ finite and \red{pseudo-periodic $\tilde u_0\in C^{2,\alpha}(\R^d, \R)$} for some $\alpha\in(0,1)$,  there exist some finite constants $C_1=C_1(d,T)$, $C_2=C_2(\alpha, \tilde u_0, T)$ and $\ep_0\in(0,1)$, such that for any $\ep\in(0,\ep_0)$, $n\in\N$ and $\delta\in(0,1)$, we have
\begin{align*}
\P\biggl(\sup_{x_i\in V,\, t\in[0,T]}|u^n(t, x_i) - u(t, x_i)|>\delta+C_2\ep^\alpha\biggr)\le C_1n^{3}e^{-C_1^{-1}\ep^{d+2} n \delta^2}+2\P\left(\|u^n_0 - u_0\|_{L^\infty(V)} > \delta/4 \right).
\end{align*}
This bound can be seen by combining Proposition \ref{prop.convergence} and \eqref{summable}. 
\end{remark}

As a consequence we obtain evidence of the existence of patterns. In particular, we prove that the system \eqref{eq:km} remains close to a continuous twisted state for times as large as we want by taking $n$ large enough (depending on the time interval and with high probability).

\begin{corollary}\label{cor}
Fix any $k\in\Z$ and $1\le\ell\le d$. Let $\mathcal A_{k,\ell}$ denote the domain of attraction, with respect to $L^\infty$-norm (induced by geodesic distance on the circle), of the continuous twisted state $u_{k,\ell}$ for the heat equation \eqref{eq:heat}, and consider any $u_0\in\mathcal A_{k,\ell}$. For any $\eta>0$ there exists $T_0=T_0(\eta, u_0)<\infty$ such that for any $T\ge T_0$, the solution $u^n$ of the Kuramoto equation with initial condition $u_0$ satisfies $$\limsup_{n\to\infty}\left\{\sup_{x_i\in V}|u^n(T,x_i)-u_{k,\ell}(x_i)|\right\}<\eta, \quad a.s.$$
\end{corollary}

\begin{proof}
Since $u_0\in\mathcal A_{k,\ell}$, the solution of the heat equation $\bar u(t)= u(t, \cdot)$ with initial condition $u_0$ converges as $t\to\infty$ to the continuous twisted state $u_{k,\ell}$, in $L^\infty(\mathbb T^d, \T)$. Given any $\eta>0$, we can choose $T_0=T_0(\eta, u_0)$ large enough such that $\|\bar u(T)-u_{k,\ell}\|_{L^\infty(V)}<\eta$ for all $T\ge T_0$. By Theorem \ref{thm.main}, for any $T\ge T_0$ and solution of the Kuramoto equation $u^n$ starting at $u_0$, we have that
\begin{align*}
&\limsup_{n\to\infty}\left\{\sup_{x_i\in V}|u^n(T,x_i)-u_{k,\ell}(x_i)|\right\}\\
&\le \limsup_{n\to\infty}\sup_{t\in[0,T]}\left\{\sup_{x_i\in V}|u^n(t,x_i)-\bar u(t, x_i)|\right\}+\|\bar u(T)-u_{k,\ell}\|_{L^\infty(V)}<\eta,
\end{align*}
almost surely.
\end{proof}
We conjecture that with high probability (as $n\to\infty$) for each $k,\ell$ there is a stable equilibrium of \eqref{eq:km} which is close to $u_{k,\ell}$. We are not able to prove this, but Theorem \ref{thm.main} can be seen as evidence to support this conjecture. This conjecture has been proved in dimension $d=1$ in \cite{devita2024energy}.


Our strategy of proof is similar in spirit to that in \cite{medvedev2018continuum} in the sense that we also consider an intermediate equation which is deterministic and behaves like solutions of \eqref{eq:km} on average (see \eqref{eq.integral}). Then we compare this intermediate solution on the one hand with the solution of \eqref{eq:km} (i.e we show that random solutions are close to their averaged equation) and on the other hand with the limiting heat equation to conclude the proof. Our averaging procedure and the way in which these two steps are carried differ largely from \cite{medvedev2018continuum}. For example, the average in \cite{medvedev2018continuum} is respect to the randomness (i.e.~taking expectation) while ours is in space. In this sense, we are closer to \cite{calder2019game, trillos2020error, trillos2016continuum}.

\red{
\begin{remark}
    From a scaling limits point of view (and moving away from the Kuramoto context), it may seem a bit ``trivial" that our limit equation is the heat equation, which either could be read as good news or bad news. On the one hand, obtaining such a simple equation helps to easily identify the behavior of the system in this scale. On the other hand, one could aim for more interesting objects in the limit. One way to modify our problem so that the sinusoidal nonlinearity persists in the limit is to replace in \eqref{eq:km} the nonlinearity $\sin \left(u^{n}(t,x_j)-u^{n}(t,x_i)\right)$ with $\sin \left(\ep^{-1}\left(u^{n}(t,x_j)-u^{n}(t,x_i)\right)\right)$. Under this scaling, one can no longer linearize the sine function because its argument $\ep^{-1}\left(u^{n}(t,x_j)-u^{n}(t,x_i)\right)$ is not small (it is of order $1$ if $u^n$ is smooth). Identifying the limit equation, particularly the form of the differential operator, appears to be an intriguing and nontrivial problem \footnote{We thank an anonymous referee for prompting this remark.} Another instance in which the nonlinearity (as given in \eqref{eq:km}) may not vanish and appears to be interesting from the point of view of understanding Kuramoto dynamics, is when the initial conditions do not converge to a smooth function. A prototypical example would be to consider i.i.d. uniform random variables as initial condition. Understanding this situation is important for the study of global synchronization versus existence of patterns (see the discussion in the introduction of \cite{DeVita2025})
\end{remark}}

The paper is organized as follows: in Section \ref{sec.integral} we prove the necessary results for the intermediate equation \eqref{eq.integral}, namely existence and uniqueness of solutions, a comparison principle, a uniform Lipschitz estimate and uniform convergence of solutions to solutions of the heat equation; in Section \ref{sec.microtoint} we prove the uniform convergence of solutions of the microscopic model to solutions of the integral equation \eqref{eq.integral} almost surely, which concludes the proof of Theorem \ref{thm.main}; {in Section \ref{sec.simulations} we discuss the relation of our results with the possible existence of twisted states. We also show some simulations to support our conjecture (in addition to our Corollary \ref{cor}) and to illustrate the main result.}

\section{The integral equation}\label{sec.integral}

In this section, we focus on \red{finding a pseudo-periodic function $u^{I,\ep}:[0,\infty)\times\R^d\to\R$  that uniquely solves} the auxiliary integral equation 
\begin{eqnarray}\label{eq.integral}
\begin{cases}
\displaystyle{\frac{d}{dt}}u^{I,\ep}(t,x)=\frac{1}{\sigma_d\ep^{d+2}}\int_{\red{\R^d}}\sin\left(u^{I,\ep}(t,y)-u^{I,\ep}(t,x)\right)K\Big(\frac{y-x}{\ep}\Big)dy, \label{eq:int}\\[10pt]
u^{I,\ep}|_{t=0}=\red{\tilde u_0,}
\end{cases}
\end{eqnarray}
where $K$ is defined as in Section \ref{subsec:setting} and \red{$\tilde u_0$ is the lift of some $u_0:\T^d\to\T$ hence pseudo-periodic.} We remark that while $u^{I,\ep}$ depends on $\ep$, \red{throughout this whole section $\ep$ is fixed}. 
We will prove:
\begin{enumerate}
\item existence and uniqueness of solutions for every fixed $\ep$ and \red{pseudo-periodic $\tilde u_0\in C^1(\R^d,\R)$};
\item comparison principles;
\item a spatial Lipschitz estimate {\it{uniform in $\epsilon$}} (assuming \red{$\tilde u_0\in C^1(\R^d,\R)$});
\item uniform convergence of solutions of \eqref{eq.integral} to solutions of the heat equation \eqref{eq:heat} as $\ep \to 0$ (assuming \red{$\tilde u_0\in C^{2,\alpha}(\R^d,\R)$} for some $\alpha\in(0,1)$).
\end{enumerate}

\subsection{Existence}

The following proposition gives existence and uniqueness of solutions to \eqref{eq.integral}. 
We will follow a fixed point procedure; let us integrate \eqref{eq:int} with respect to time to get
\begin{equation}\label{eq.fixedpoint}
u^{I,\ep}(t,x) =\tilde u_0(x) + \frac{1}{\sigma_d\ep^{d+2}}\int_0^t\int_{\red{\R^d}}\sin\left(u^{I,\ep}(s,y) - u^{I,\ep}(s,x)\right)K\Big(\frac{y-x}{\ep}\Big)dy\:ds.
\end{equation}
We see that finding a solution of the integral equation \eqref{eq:int} is equivalent to finding \red{a pseudo-periodic $u^{I,\ep}:[0,\infty)\times\R^d\to\R$ with}
\[
\red{u^{I,\ep}\in C\big([0,\infty);C^1(\R^d, \R)\big)}
\]
satisfying \eqref{eq.fixedpoint}. 

\begin{lemma}\label{lem:derivative}
Let $F(x):=\int_{\red{\R^d}}g(x,y)\mathsf K(y-x)dy$, where 
$g(x,y)\in C^1_b(\R^d\times\R^d,\R)$ \red{(with bounded $C^1$-norm)} and $\mathsf K\in L^1(\R^d,\R)$ with compact support, then the following derivative formula holds:
\begin{align*}
\nabla F(x)=\int_{\red{\R^d}}\left[\nabla_x g(x,y)+\nabla_yg(x,y) \right] \mathsf K(y-x) dy.
\end{align*}
\end{lemma}

\begin{proof}
Consider a sequence of $C_c^\infty(\R^d)$ functions $\{\mathsf K_n: \R^d\to\R\}_{n\in\N}$ such that $\mathsf K_n\to\mathsf K$ in $L^1(\R^d)$. Let 
\begin{align}\label{approx-fct}
F_n(x):=\int_{\R^d}g(x,y)\mathsf K_n(y-x)dy
\end{align}
and clearly we have that 
\begin{align}
\nabla F_n(x)&=\int_{\R^d}\left[\nabla_x g(x,y)\mathsf K_n(y-x)+g(x,y)\nabla_x (\mathsf K_n(y-x))\right] dy  \nonumber\\
&=\int_{\R^d}\left[\nabla_x g(x,y)\mathsf K_n(y-x)-g(x,y)\nabla_y (\mathsf K_n(y-x))\right] dy   \nonumber\\
&=\int_{\R^d}\left[\nabla_x g(x,y)+\nabla_yg(x,y) \right] \mathsf K_n(y-x) dy,   \label{formula-derivative}
\end{align}
where the last step is due to integration by parts, \red{and due to $\mathsf K_n$ being compactly supported, there are no boundary terms}. Since $g\in C^1_b$ (with bounded $C^1$-norm), we have for any $m,n\in\N$ that 
\begin{align*}
\sup_{x\in\R^d}|\nabla F_n(x)-\nabla F_m(x)|\le 2\|g\|_{C^1(\R^{2d})}\|\mathsf K_n- \mathsf K_m\|_{L^1(\R^d)}.
\end{align*}
We also have that 
\begin{align*}
\sup_{x\in\R^d}|F_n(x)-F_m(x)|\le \|g\|_{L^\infty(\R^{2d})}\|\mathsf K_n- \mathsf K_m\|_{L^1(\R^d)}.
\end{align*}
In other words,
\begin{align*}
\|F_n-F_m\|_{C^1(\R^d)}\le 3\|g\|_{C^1(\R^{2d})}\|\mathsf K_n-\mathsf K_m\|_{L^1(\R^d)},
\end{align*}
where the $C^1$-norm of $g$ is defined as usual  $\|g\|_{C^1}:=\|g\|_{L^\infty}+\|\nabla_x g\|_{L^\infty}+\|\nabla_yg\|_{L^\infty}$. 

Since $\{\mathsf K_n\}_{n\in\N}$ is Cauchy in $L^1(\R^d)$, we have that $\{F_n\}_{n\in\N}$ is Cauchy in $C^1(\R^d)$ and we call $\lim_{n\to\infty}F_n=:\wt F$ the limit. Further, since $\mathsf K$ is the $L^1$-limit of $\mathsf K_n$, passing to $n\to\infty$ on both sides of \eqref{approx-fct} and comparing with the formula of $F$, we see that $\wt F=F$. Then passing to $n\to\infty$ on both sides of \eqref{formula-derivative} we see that the claimed formula for $\nabla F$ holds.
\end{proof}

\red{
\begin{lemma}\label{lem:periodic}
Suppose that $u(t,x):[0,\infty)\times\R^d\to\R$ is a continuous function in $(t,x)$, pseudo-periodic in $x$ for every $t$, then the function
\[
U(t,x):=\int_{\R^d}\sin\left(u(t,y) - u(t,x)\right)K\Big(\frac{y-x}{\ep}\Big)dy
\]
is $2\pi$-periodic in $x$ in each coordinate direction, for every $t$.
\end{lemma}
\begin{proof}
    Let $\{k_\ell\}_{\ell=1}^d \in\Z^d$ denote the winding numbers of $u(t,\cdot)$ for every $t\ge0$ (which by Lemma \ref{lem:winding} are time-independent and determined by $u|_{t=0}$). We have, for every $\ell=1,2,...,d$ and $x\in\R^d$,
    \begin{align*}
       U(t,x+2\pi \vec e_\ell) &=\int_{\R^d}\sin\left(u(t,y) - u(t,x+2\pi \vec e_\ell)\right)K\Big(\frac{y-x-2\pi \vec e_\ell}{\ep}\Big)dy\\
        &=\int_{\R^d}\sin\left(u(t,y'+2\pi \vec e_\ell) - u(t,x+2\pi \vec e_\ell)\right)K\Big(\frac{y'-x}{\ep}\Big)dy'\\
        &=\int_{\R^d}\sin\left(u(t,y')+2\pi k_\ell - u(t,x)-2\pi k_\ell\right)K\Big(\frac{y'-x}{\ep}\Big)dy'\\
        &=\int_{\R^d}\sin\left(u(t,y') - u(t,x)\right)K\Big(\frac{y'-x}{\ep}\Big)dy'=U(t,x),
    \end{align*}
    where in the second equality we made a change of variables $y=y'+2\pi \vec e_\ell$, and in the third equality we used the pseudo-periodicity of $u(t,\cdot)$. This completes the proof.
\end{proof}
}

\begin{proposition}\label{prop.existence}
Fix any $\ep>0$. Given \red{any pseudo-periodic $\tilde u_0\in C^1(\R^d,\R)$}, there exists a unique \red{pseudo-periodic function $u^{I,\ep}\in C\left([0,\infty);C^1(\R^d,\R)\right)$} satisfying \eqref{eq.fixedpoint} and hence a unique solution of \eqref{eq.integral}.
\end{proposition}

\begin{proof}

Solutions of \eqref{eq.fixedpoint} are fixed points of the operator
\begin{align}\label{fixed-pt-map}
F_{\tilde u_0}(u^{I,\ep})(t,x) :=\tilde  u_0(x) + \frac{1}{\sigma_d\ep^{d+2}}\int_0^t\int_{\red{\R^d}}\sin\left(u^{I,\ep}(s,y) - u^{I,\ep}(s,x)\right)K\Big(\frac{y-x}{\ep}\Big)dy\:ds.
\end{align}
For a fixed \red{pseudo-periodic} initial condition \red{$\tilde u_0\in C^1(\R^d,\R)$} and a positive $T$ (to be chosen later) we consider the Banach space 
\red{\begin{align}\label{banach-ball}
\mathcal{X}_T:=\Big\{f: [0,T]\times\R^d\to\R:\; & f(t,\cdot)\text{ is pseudo-periodic},\, \, f\in C\big([0,T];C^1(\R^d,\R)\big), \nonumber\\
&f|_{t=0}=\tilde u_0, \;  \sup_{t\in[0,T]}\|f(t,\cdot)\|_{C^1(\Lambda_d)}\le 1+\|\tilde u_0\|_{C^1(\Lambda_d)}\Big\}
\end{align}
with the norm 
\[
\|f\|_{\mathcal{X}_T} := \sup_{t\in[0,T]}\|f(t,\cdot)\|_{C^1(\Lambda_d)},
\]
where by an abuse of notation, for a pseudo-periodic function $g:\R^d\to\R$, we write $\|g\|_{C^1(\Lambda_d)}$ for $\|g|_{\Lambda_d}\|_{C^1(\Lambda_d)}$. We immediately remark that the map $F_{\tilde u_0}$ \eqref{fixed-pt-map} takes a pseudo-periodic continuous function to another pseudo-periodic continuous function, with the same winding numbers. This is a consequence of Lemma \ref{lem:periodic} and that $\tilde u_0$ is pseudo-periodic.
}

We want to apply Banach's fixed point theorem  to $F_{\tilde u_0}$ in $\mathcal{X}_T$. We must check:

\begin{enumerate}[label=(\alph*)]

\item $F_{\tilde u_0}(\mathcal{X}_T)\subset\mathcal{X}_T$;

\item $F_{\tilde u_0}$ is a contraction, i.e.~there exists $\nu\in(0,1)$ such that 
\begin{equation*}
\|F_{\tilde u_0}(u^{I,\ep})-F_{\tilde u_0}(v)\|_{\mathcal{X}_T}\leq \nu\|u-v\|_{\mathcal{X}_T}
\end{equation*}
for all $u, v\in \mathcal{X}_T$.
\end{enumerate}

We first show (b). \red{We first observe that $h(x,y):=\sin\left(u(t,y)-u(t,x)\right)$ has bounded $C^1(\R^{2d})$-norm for each $t$. Indeed, $\|h\|_{L^\infty(\R^{2d})}\le 1$ and furthermore, $$\nabla_yh(x,y)=\cos\left(u(t,y)-u(t,x)\right)\nabla_y u(t,y).$$ Since $u(t,y)$ is pseudo-periodic, $\|\nabla_y u(t,y)\|_{L^\infty(\R^d)}$ is equal to $\|\nabla_y u(t,y)\|_{L^\infty(\Lambda_d)}$, and therefore is bounded. Hence $\|\nabla_yh(x,y)\|_{L^\infty(\R^{2d})}$ is bounded, and similarly $\|\nabla_xh(x,y)\|_{L^\infty(\R^{2d})}$ is bounded as well.}
We are in a situation to apply Lemma \ref{lem:derivative}, whereby for any $x\in\Lambda_d$,
\begin{align*}
& \sigma_d\ep^{d+2}|\nabla F_{\tilde u_0}(u)(t,x)-\nabla F_{\tilde u_0}(v)(t,x)| \\
& \le \int_0^t\int_{\R^d}\Big|\cos\left(u(s,y)-u(s,x)\right)\nabla u(s,y)-\cos\left(v(s,y)-v(s,x)\right)\nabla v(s,y)\Big|K\Big(\frac{y-x}{\ep}\Big)dy\:ds  \\
&+ \int_0^t\int_{\R^d}\Big|\cos\left(u(s,y)-u(s,x)\right)\nabla u(s,x)-\cos\left(v(s,y)-v(s,x)\right)\nabla v(s,x)\Big|K\Big(\frac{y-x}{\ep}\Big)dy\:ds.
\end{align*}
\red{
We notice that although we consider $x\in\Lambda_d$, there may be some $y\in\R^d$ that are within Euclidean distance $\ep$ from $x$ but are outside of $\Lambda_d$. Since $\ep\in(0,1)$, such $y$ must be in one of the adjacent boxes. We must have either $y_\ell=[y]_\ell\pm 2\pi$ or $y_\ell=[y]_\ell$, $\ell=1,2,...,d$ (see \eqref{def:mod}). For $g$ that is pseudo-periodic with winding numbers $\{k_\ell\}_{\ell=1}^d$, we have $g([y])=g(y)+\sum_{\ell=1}^da_\ell k_\ell \vec e_\ell$ for some $a_\ell\in\{0,\pm 2\pi\}$ and $\nabla g([y])=\nabla g(y)$. Hence, since the cosine is a $2\pi$-periodic function and $u,v$ are pseudo-periodic, the previous display equals
\begin{align*}
& =\int_0^t\int_{\R^d}\Big|\cos\left(u(s,[y])-u(s,x)\right)\nabla u(s,[y])-\cos\left(v(s,[y])-v(s,x)\right)\nabla v(s,[y])\Big|K\Big(\frac{y-x}{\ep}\Big)dy\:ds  \\
&+ \int_0^t\int_{\R^d}\Big|\cos\left(u(s,[y])-u(s,x)\right)\nabla u(s,x)-\cos\left(v(s,[y])-v(s,x)\right)\nabla v(s,x)\Big|K\Big(\frac{y-x}{\ep}\Big)dy\:ds.
\end{align*}}

Then by the triangle inequality and $1$-Lipschitz property of sine function, we further bound it by
\begin{align*}
& \le \int_0^t\int_{\R^d}\Big(|\nabla u(s,[y])-\nabla v(s,[y])| \\
&\hspace{4cm}+\big(|u(s,[y])-v(s,[y])|+|u(s,x)-v(s,x)|\big)|\nabla v(s,[y])|\Big)K\Big(\frac{y-x}{\ep}\Big)dy\:ds\\
&+\int_0^t\int_{\R^d}\Big(|\nabla u(s,x)-\nabla v(s,x)| \\
&\hspace{4cm}+\big(|u(s,[y])-v(s,[y])|+|u(s,x)-v(s,x)|\big)|\nabla v(s,x)|\Big)K\Big(\frac{y-x}{\ep}\Big)dy\:ds.
\end{align*}

Since $\|K\|_{L^1(\R^d)}=1$, applying \red{H\"older's inequality $$\|fK(\frac{y - \cdot}{\ep})\|_{L^1(\R^d)}\le \|f\|_{L^\infty(\R^d)}\|K(\frac{y-\cdot}{\ep})\|_{L^1(\R^d)} = \ep^d\|f\|_{L^\infty(\R^d)},$$ to the above integrals in $dy$, and noting that both $x,[y]\in\Lambda_d$}, we get from above
\begin{align}\label{eq:grad-diff}
&\sup_{t\in[0,T]}\|\nabla F_{\tilde u_0}(u)(t,\cdot)-\nabla F_{\tilde u_0}(v)(t,\cdot)\|_{L^\infty(\Lambda_d)}\\
&\le C\left(1+\|\tilde u_0\|_{C^1(\Lambda_d)}\right)T \sup_{s\in[0,T]}\|u(s,\cdot)-v(s,\cdot)\|_{L^\infty(\Lambda_d)}+CT \sup_{s\in[0,T]}\|\nabla u(s,\cdot)-\nabla v(s,\cdot)\|_{L^\infty(\Lambda_d)}, \nonumber
\end{align} 
where the constant $C=C(\ep, d)$ may change from line to line, and we used that since $v\in \cX_T$ \eqref{banach-ball}, its $C^1(\Lambda_d)$-norm in space is bounded by $1+\|\tilde u_0\|_{C^1(\Lambda_d)}$ up to time $T$.

Similarly, we can estimate 
\begin{align*}
&\sigma_d\ep^{d+2} |F_{\tilde u_0}(u)(t,x)-F_{\tilde u_0}(v)(t,x)| \\
& \leq \int_0^t\int_{\R^d}\Big|\sin\left(u(s,y)-u(s,x)\right)-\sin\left(v(s,y)-v(s,x)\right)\Big|K\Big(\frac{y-x}{\ep}\Big)dy\:ds \\
&= \red{\int_0^t\int_{\R^d}\Big|\sin\left(u(s,[y])-u(s,x)\right)-\sin\left(v(s,[y])-v(s,x)\right)\Big|K\Big(\frac{y-x}{\ep}\Big)dy\:ds,}
\end{align*}
which yields that
\begin{align}\label{eq:diff}
&\sup_{t\in[0,T]}\|F_{\tilde u_0}(u)(t,\cdot)-F_{\tilde u_0}(v)(t,\cdot)\|_{L^\infty(\Lambda_d)}\le C(\ep, d)T \sup_{t\in[0,T]}\|u(s,\cdot)-v(s,\cdot)\|_{L^\infty(\Lambda_d)}.
\end{align}
Recalling the definition of $C^1(\Lambda_d)$-norm, we get from \eqref{eq:grad-diff}-\eqref{eq:diff}
\begin{align*}
\|F_{\tilde u_0}(u)-F_{\tilde u_0}(v)\|_{\cX_T}\le CT\left(1+\|\tilde u_0\|_{C^1(\Lambda_d)}\right)\|u-v\|_{\cX_T},
\end{align*}
for some finite $C=C(\ep, d)$. 

For requirement (a), a similar argument (just do not consider $v$) shows that 
\begin{align*}
\|F_{\tilde u_0}(u)\|_{\cX_T}\le \|\tilde u_0\|_{C^1(\Lambda_d)}+CT\|u\|_{\cX_T}.
\end{align*}

Choosing $T_0:=\frac{1}{2C\left(1+\|\tilde u_0\|_{C^1(\Lambda_d)}\right)}$ we get by Banach's fixed point theorem local existence of a unique solution to \eqref{eq.fixedpoint} in the time interval $[0,T_0]$. We want to iterate the above argument and get global existence. We note that even though $T_0$ depends inversely on $C^1$-norm of the initial condition, the dependence is linear. On the other hand, our construction \eqref{banach-ball} is such that each iteration only increases the norm of the solution by at most $1$. This means that, if we iterate the previous argument infinitely many times, the sum of the local existence intervals diverges and we can reach any $T$. Finally note that a continuous solution of \eqref{eq.fixedpoint} is in fact differentiable in $t$ and as such, a solution of \eqref{eq.integral}.
\end{proof}

\subsection{Comparison Principle and Lipschitz Estimate}

This section is devoted to the proof of a comparison principle for \eqref{eq.integral} and the consequential uniform Lipschitz estimate. 
We start with the comparison principle.

\begin{lemma}\label{lem:max}
Fix $T$ finite. Let $\Psi: [0,\infty)\times\R^d\times\R^d\to\R_+$ be strictly positive and $v, w:[0, T]\times\R^d\to\R$ be two continuous functions,  with continuous time derivative, \red{$2\pi$-periodic in $x$} that satisfy
\begin{align}
\frac{d}{dt}v(t,x) - \mathcal L v (t,x) &\geq \frac{d}{dt}w(t,x) - \mathcal L w (t,x), \label{eq:comparison1} \\[4pt]
v(0,x)&\geq  w(0,x), \quad x\in\Lambda_d,\, t\in[0,T],  \label{eq:comparison2}
\end{align}
with 
\[
\mathcal L v (t,x) := \frac{1}{\sigma_d\ep^{d+2}}\int_{\R^d}\Psi(t, x,y)\left(v(t,y)-v(t,x)\right) K\Big(\frac{y-x}{\ep}\Big)dy.
\]
Then,
\[
v(t,x)\geq w(t,x), \quad  x\in\Lambda_d, \, t\in[0,T].
\]
\end{lemma}

\begin{proof}
Assume that the conclusion of the lemma fails to hold, i.e.~there exists $(t_0,x_0)\in[0,T]\times\Lambda_d$ (not necessarily in the interior), 
such that
\[
v(t_0,x_0)<w(t_0,x_0).
\]
We may assume that such a point is a minimum point for $v-w$ on $[0,T]\times\Lambda_d$ since $v,w$ are continuous. Notice in particular that $t_0>0$, since $(v-w)|_{t=0}\ge 0$. By \eqref{eq:comparison1},
\[
\frac{d}{dt}(v-w)(t_0,x_0)\geq \frac{1}{\sigma_d\ep^{d+2}}\int_{\R^d}\Psi(t, x_0, y)\left((v-w)(t_0,y)-(v-w)(t_0,x_0)\right) K\Big(\frac{y-x}{\ep}\Big)dy.
\]
But since $(t_0,x_0)$ is a minimum for $v-w$ on $[0,T]\times\Lambda_d$ and $\Psi$ is positive, we have
\begin{align*}
&\frac{1}{\sigma_d\ep^{d+2}}\int_{\R^d}\Psi(t, x_0, y)\left((v-w)(t_0,y)-(v-w)(t_0,x_0)\right) K\Big(\frac{y-x}{\ep}\Big)dy\\
&=\frac{1}{\sigma_d\ep^{d+2}}\int_{\R^d}\Psi(t, x_0, y)\left((v-w)(t_0,[y])-(v-w)(t_0,x_0)\right) K\Big(\frac{y-x}{\ep}\Big)dy> 0,
\end{align*}
\red{where we used the $2\pi$-periodicty of $v-w$ and $[y]\in\Lambda_d$ as defined in \eqref{def:mod}, }
and
\[
\frac{d}{dt}(v-w)(t_0,x_0)\le 0,
\]
thus we have a contradiction.
\end{proof}

From Lemma \ref{lem:max} we can deduce a Lipschitz estimate uniformly in $\ep$. Let us define
\[
\red{\|f\|_{\rm Lip(\Lambda_d)}:=\sup_{x\in\Lambda_d,\, h\in\R^d\backslash\{0\}}\frac{|f(x+h)-f(x)|}{|h|}.}
\]

\begin{lemma}\label{lem:lip}
Let $T\in(0,\infty)$ be fixed and \red{$\tilde u_0\in C^1(\R^d,\R)$ be pseudo-periodic}. Then there exist $\ep_0=\ep_0(T,\tilde u_0)>0$ and $C_T=C_T(\tilde u_0)$ finite such that if $\ep<\ep_0$, $u^{I,\ep}$ is the \red{unique pseudo-periodic solution of \eqref{eq.integral} from Proposition \ref{prop.existence},} then for every $0\le t \le T$,
\[
\|u^{I,\ep}(t,\cdot)\|_{\rm Lip(\Lambda_d)}\le C_T.
\]
\end{lemma}
\begin{remark}
The uniform Lipschitz estimate will be used in the next Section \ref{sec.microtoint}, but we point out that it is not needed for the proof of convergence of solutions of the integral equation to solutions of the heat equation. 
\end{remark}

\begin{proof}

Fix $h\in\R^d\backslash\{0\}$ and consider $w_h(t,x):=u^{I,\ep}(t,x+h)-u^{I,\ep}(t,x)$. \red{Since $u^{I,\ep}(t,\cdot)$ is pseudo-periodic, we see that $w_h(t,\cdot)$ is $2\pi$-periodic in each coordinate direction. This follows from \eqref{eq:pseudo-periodic}.} Using \eqref{eq:int}, we have that $w_h$ satisfies for every $x\in\Lambda_d$
\begin{align*}
\frac{d}{dt}&w_h(t,x)\\
=&\displaystyle\frac{1}{\sigma_d\ep^{d+2}}\int_{\R^d}\left(\sin\left(u^{I,\ep}(t,y+h)-u^{I,\ep}(t,x+h)\right)-\sin\left(u^{I,\ep}(t,y)-u^{I,\ep}(t,x)\right)\right)K\Big(\frac{y-x}{\ep}\Big)dy \:dy,
\end{align*}
by a change of variables, with $ w_h|_{t=0}=\tilde u_0(x+h)-\tilde u_0(x)$.
Note that we can rewrite
\begin{align*}
&\sin\left(u^{I,\ep}(t,y+h)-u^{I,\ep}(t,x+h)\right)-\sin\big(u^{I,\ep}(t,y)-u^{I,\ep}(t,x)\big)\\
&=\Psi(t, x,y)\big(w_h(t,y)-w_h(t,x)\big),
\end{align*}
with
\begin{align*}
\Psi(t, x, y)&:=\int_0^1\cos\left(s(u^{I,\ep}(t,y+h)-u^{I,\ep}(t,x+h))+(1-s)(u^{I,\ep}(t,y)-u^{I,\ep}(t,x))\right)\:ds.
\end{align*}

For some large constant $M$ (to be specified), let us define a time 
\begin{align*}
\tau_M:=\inf\big\{t\ge0: \|u^{I,\ep}(t, \cdot)\|_{\text{Lip}(\Lambda_d)}\ge M\big\}.
\end{align*}
We claim that, by taking $\ep$ small enough (for instance $\ep<\pi/(4M)$), we can ensure that $\Psi(t, x, z)\geq c_0>0$ for some constant $c_0=c_0(M)$, for any $t\in[0,\tau_M)$. Indeed, \red{the control of $\|u^{I,\ep}(t, \cdot)\|_{\text{Lip}(\Lambda_d)}$ up to $\tau_M$, together with the pseudo-periodicity of $u^{I,\ep}(t,\cdot)$, imply that for any $y$ with $|y-x|\le\ep$,
\begin{align*}
|u^{I,\ep}(t,y)&-u^{I,\ep}(t,x)|\le M|y-x|\le M\ep,\\
|u^{I,\ep}(t,y+h)&-u^{I,\ep}(t,x+h)|\\
=&|u^{I,\ep}(t, y-x+[x+h])-u^{I,\ep}(t,[x+h])|\le M|y-x|\le M\ep.
\end{align*}
In the last line, the difference of $u^{I,\ep}(t,\cdot)$ at two different spatial points $y+h, x+h$ is unchanged under a simultaneous translation by $[x+h]-(x+h)$ (which is a multiple of $2\pi$ in each coordinate direction), due to pseudo-periodicity of $u^{I,\ep}(t,\cdot)$, see \eqref{eq:pseudo-periodic}.} The argument of the cosine in $\Psi(x,z )$ is thus uniformly close to $0$. 

To apply the maximum principle, Lemma \ref{lem:max}, we next use the following barrier
\[
\bar w(t,x):=(t+1)|h|\|\tilde u_0\|_{\rm Lip(\Lambda_d)}, \quad t\ge0, \, x\in\R^d
\]
(notice that this function is space independent). The function $\bar w$ satisfies
\[
\frac{d}{dt}\bar w (t,x)- \frac{1}{\sigma_d\ep^{d+2}}\int_{\R^d}\Psi(t, x, y)\left(\bar w(t, y)-\bar w(t, x)\right) K\Big(\frac{y-x}{\ep}\Big)dy =\frac{d}{dt}\bar w= |h|\|u_0\|_{\rm Lip}\ge 0,
\]
and as shown above, $w_h$ satisfies
\[
\frac{d}{dt}w_h(t,x)-\frac{1}{\sigma_d\ep^{d+2}}\int_{\R^d}\Psi\left(t, x, y\right)\left(w_h(t, y)-w_h(t, x)\right) K\Big(\frac{y-x}{\ep}\Big)dy dy=0.
\]
Also, $\bar w(0,x)\geq  w_h(0,x)$ for $x\in\Lambda_d$ since due to $u^{I,\ep}|_{t=0}=\tilde u_0\in C^1(\R^d)$, we have
\[
w_h(0,x)\le |w_h(0,x)|=|\tilde u_0(x+h)-\tilde u_0(x)|\le \|\tilde u_0\|_{\rm Lip(\Lambda_d)}|h|=\bar w(0,x).
\]
Hence, \red{with both functions $w_h, \bar w$ continuous and $2\pi$-periodic,} we satisfy the conditions of Lemma \ref{lem:max} which yields
\begin{align*}
w_h(t,x)\leq \bar w(t,x)= (t+1) |h|\|u_0\|_{\rm Lip(\Lambda_d)},\quad t\in[0,\tau_M),\, x\in\Lambda_d.
\end{align*}
An analogous reasoning for $-w_h$ gives 
\begin{align*}
-w_h(t,x)\leq (t+1) |h|\|\tilde u_0\|_{\rm Lip(\Lambda_d)},\quad t\in[0,\tau_M),\, x\in\Lambda_d.
\end{align*}
In other words, for $t\in[0,\tau_M)$,
\begin{align}\label{heat-comparison}
\sup_{x\in\Lambda_d}\frac{|u^{I,\ep}(t,x+h)-u^{I,\ep}(x)|}{|h|}\le (t+1)\|\tilde u_0\|_{\rm Lip(\Lambda_d)}.
\end{align}
Choosing $M=2(T+1)\|\tilde u_0\|_{\rm Lip(\Lambda_d)}$ then ensures that $\tau_M> T$, and hence \eqref{heat-comparison} holds for any $t\in[0,T]$. Since the right-hand-side of \eqref{heat-comparison} does not depend on $h$, we obtain the desired Lipschitz bound with $C_T:=(T+1)\|\tilde u_0\|_{\rm Lip(\Lambda_d)}$.
\end{proof}

\subsection{Convergence to solutions of the heat equation}

Next we show that solutions to the integral equation \eqref{eq:int} converge to solutions of the heat equation \eqref{eq:heat} as $\ep \to 0$. Notice that we are assuming the same initial condition for the heat and integral equations.
\begin{proposition}\label{prop.convergence}
Let $u:[0,\infty)\times\R^d\to\R$ be the \red{unique pseudo-periodic} solution of \eqref{eq:heat} and $u^{I,\ep}:[0,\infty)\times \R^d\to\R$ the \red{unique  pseudo-periodic solution of \eqref{eq.integral}, with the same pseudo-periodic initial condition $u|_{t=0}=u^{I,\ep}|_{t=0}=\tilde u_0\in C^{2,\alpha}(\R^d,\R)$} for some $\alpha\in(0,1)$. Then for any $T>0$ there exists $C=C(T, \alpha, \tilde u_0, K)>0$ and $\ep_0\in(0,1)$ such that for any $\ep<\ep_0$,
\[
\sup_{t\in[0,T]}\|u^{I,\ep}(t,\cdot)-u(t,\cdot)\|_{L^\infty(\Lambda_d)}\leq C\ep^\alpha.
\]
\end{proposition}

Before proving Proposition \ref{prop.convergence}, we show pointwise convergence of the integral operator to the Laplacian.
\begin{lemma}\label{lem:integraltoheat}
\red{Let $v\in C^{2,\alpha}(\R^d, \R)$ be pseudo-periodic, for some  $\alpha\in(0,1)$, then there exists a finite constant $\tilde C$ depending only on $K, d$ and the $C^{2,\alpha}(\Lambda'_d)$-norm of $v$ such that for any $\ep\in(0,1)$
\[
\sup_{x\in\Lambda_d}\Big|\frac{1}{\sigma_d\ep^{d+2}}\int_{\R^d}\sin\left(v(y)-v(x)\right)K\Big(\frac{y-x}{\ep}\Big)dy-\frac{\kappa_2}{2d}\Delta v(x)\Big|\le \tilde {C}\ep^\alpha,
\]
where $\Lambda'_d:=[-1,2\pi+1]^d$ is a $1$-enlargement of $\Lambda_d$.}
\end{lemma}

\begin{proof}
Note that for all $t\in\R$ it holds
\[
|\sin t-t|\le t^3.
\]
Indeed, let us check this for $t\ge0$, and the case $t<0$ follows since the involved functions are odd. Let $h(t):=t-\sin t-t^3$ which satisfies $h(0)=0$ and $h'(t)=1-\cos t-3t^2=2\sin^2\frac{t}{2}-3t^2\le \frac{t^2}{2}-3t^2\le 0$, and thus $h(t)\le h(0)=0$ for all $t\ge0$. Let us denote $E(t):=\sin t -t$ and we have
\begin{align*}
&\frac{1}{\sigma_d\ep^{d+2}}\int_{\R^d}\sin\left(v(y)-v(x)\right)K\Big(\frac{y-x}{\ep}\Big)dy \\
& 	= \frac{1}{\sigma_d\ep^{d+2}}\int_{\R^d}\left(v(y)-v(x)\right)K\Big(\frac{y-x}{\ep}\Big)dy +\frac{1}{\sigma_d\ep^{d+2}}\int_{\R^d}E\left(v(y)-v(x)\right)K\Big(\frac{y-x}{\ep}\Big)dy \\
& =: (A)+(B).
\end{align*}
Regarding the first term (A), since $v\in C^{2,\alpha}(\R^d)$ we have that
\[
\Big|v(y)-v(x)-D v(x)\cdot(y-x)-\frac{1}{2}D^2v(x)(y-x)\cdot(y-x)\Big|\le \tilde C|x-y|^{2+\alpha}
\]
for all $x\in\Lambda_d$, $|y-x|\le 1$ and some finite constant $\tilde C$ that depends on the $C^{2,\alpha}(\Lambda'_d)$-norm of $v$. Hence,
\begin{align*}
\Big|(A) & -\frac{1}{\sigma_d\ep^{d+2}}\left(\int_{\R^d}D v(x)\cdot(y-x)K\Big(\frac{y-x}{\ep}\Big)dy -\int_{\R^d}\frac{1}{2}D^2v(x)(y-x)\cdot(y-x)K\Big(\frac{y-x}{\ep}\Big)dy\right)\Big|\\
&\le \frac{1}{\sigma_d\ep^{d+2}}\int_{\R^d}\tilde C|x-y|^{2+\alpha}K\Big(\frac{y-x}{\ep}\Big)dy=\frac{\ep^\alpha}{\sigma_d}\int_{\R^d}\tilde C|z|^{2+\alpha}K(z)dz\\
&\le\tilde C\ep^\alpha.
\end{align*}
Since $K$ is radially symmetric, by symmetry of the integrands we have
\[
\int_{\R^d}D v(x)\cdot(y-x)K\Big(\frac{y-x}{\ep}\Big)dy=0, \quad \delta_{i\neq j}\int_{\R^d}\frac{1}{2}D_{ij}v(x)(y-x)_i(y-x)_jK\Big(\frac{y-x}{\ep}\Big)dy=0,
\]
and recalling the constant $\kappa_2$ \eqref{kappa-2} 
\[
\frac{1}{\sigma_d\ep^{d+2}}\int_{\R^d}\frac{1}{2}\sum_{i=1}^d D_{ii}v(x)(y-x)^2_i K\Big(\frac{y-x}{\ep}\Big)dy=\frac{\kappa_2}{2d}\Delta v(x).
\]
This shows that 
\[
\Big|(A)-\frac{\kappa_2}{2d}\Delta v(x)\Big|\le \tilde C\ep^\alpha.
\]
Turning to the second term (B), for any $x\in\Lambda_d$ and $|y-x|\le 1$ we have
\begin{align*}
   &|(B)|= \Big|\frac{1}{\sigma_d\ep^{d+2}}\int_{\R^d}E\left(v(y)-v(x)\right)K\Big(\frac{y-x}{\ep}\Big)dy \Big|\\
   &\le \frac{1}{\sigma_d\ep^{d+2}}\int_{\R^d}|E\left(v(y)-v(x)\right)|K\Big(\frac{y-x}{\ep}\Big)dy \le \frac{1}{\sigma_d\ep^{d+2}}\int_{\R^d}|v(y)-v(x)|^3K\Big(\frac{y-x}{\ep}\Big)dy\\
   &\le \frac{1}{\sigma_d\ep^{d+2}}\int_{\R^d}\|v\|_{C^1(\Lambda'_d)}^3|y-x|^3K\Big(\frac{y-x}{\ep}\Big)dy
   = \frac{\ep}{\sigma_d}\|v\|_{C^1(\Lambda'_d)}^3\int_{\R^d}|z|^3K(z)dz\\
   &=\tilde C \ep,
\end{align*}
for some possibly different $\tilde C$.
\end{proof}

We are ready to prove the convergence of $u^{I,\ep}$ to $u$.

\begin{proof}[Proof of Proposition \ref{prop.convergence}]
Since $\tilde u_0\in C^{2, \alpha}(\R^d,\R)$ for some $\alpha\in(0,1)$, parabolic regularity theory (Schauder estimates) guarantees that the unique solution of the heat equation $u(t,x)$ is $C^{2,\alpha}$ in space and $C^{1,\alpha/2}$ in time, cf.~\cite[Theorem 9.1.2]{krylov}. \red{More precisely, since the supremum norm of $\tilde u_0$ over $\R^d$ is infinite when the winding numbers are non-zero, to apply the standard Schauder theorem we can first subtract the ``tilt" $\tilde u_0(x)-\sum_{\ell=1}^d k_\ell x\cdot\vec e_\ell$ from the initial condition, solve the heat equation with such initial condition in the correct H\"older space, and then add back the ``tilt".}
Let 
\[
w^\ep(t,x):=u^{I,\ep}(t,x)-u(t,x). 
\]
\red{Since $u^{I,\ep}, u$ are both pseudo-periodic, continuous, with the same pseudo-periodic initial condition $\tilde u_0$, they have the same time-independent winding numbers. Hence, $w^\ep(t,\cdot)$ is $2\pi$-periodic in each coordinate direction.}
We write
\begin{align*}
&\frac{d}{dt}w^\ep(t,x)\\
&= \frac{1}{\sigma_d\ep^{d+2}}\int_{\R^d}\left(\sin\left(u^{I,\ep}(t,y)-u^{I,\ep}(t,x)\right)-\sin\left(u(t,y)-u(t,x)\right)\right)K\Big(\frac{y-x}{\ep}\Big)dy + E_\ep(u)(t,x),
\end{align*}
where
\begin{equation}\label{eq:error}
E_\ep(u)(t,x):=\frac{1}{\sigma_d\ep^{d+2}}\int_{\R^d}\sin\left(u(t,y)-u(t,x)\right)K\Big(\frac{y-x}{\ep}\Big)dy-\frac{1}{2}\kappa_2\Delta u(x). 
\end{equation}
For the pseudo-periodic solution $u(t,x)$ of the heat equation \eqref{eq:heat}, during time interval $[0,T]$ its spatial $C^{2,\alpha}(\Lambda'_d)$-norm is uniformly bounded by a constant that depends on $\tilde u_0, T, d$, and hence thanks to Lemma \ref{lem:integraltoheat}, the error term 
\begin{align}\label{bound:error}
\sup_{t\in[0,T], \, x\in\Lambda_d}|E_\ep(u)(t,x)|\le \tilde C\ep^\alpha
\end{align}
for some finite constant $\tilde C=\tilde C(\tilde u_0,T, d, K)$ independent of $\ep$. 

Now, for any $x\in\Lambda_d$, 
\begin{align*}
&\sin\left(u^{I,\ep}(t,y)-u^{I,\ep}(t,x)\right)-\sin\left(u(t,y)-u(t,x)\right)\\
&=\Psi(t, x,y)\left(w^\ep(t,y)-w^\ep(t,x)\right)
\end{align*}
where 
\begin{align*}
\Psi(t, x,y)&:=\int_0^1\cos \left(s\left(u^{I,\ep}(t,y)-u^{I,\ep}(t,x)\right)+(1-s)\left(u(t,y)-u(t,x)\right)\right)ds,
\end{align*}

Notice that, due to the uniform (in $\ep$) Lipschitz bound on $u^{I,\ep}(t,\cdot)$ (see Lemma \ref{lem:lip}) and $u(t,\cdot)$ in the time interval $[0,T]$, for $y$ with $|y-x|\le \ep$, the argument of the cosine in $\Psi(t, x,y)$ can be chosen to be uniformly close to $0$ (by making $\ep$ small enough), and $\Psi(t,x,y)$ bounded uniformly away from $0$.

We want to apply the comparison principle Lemma \ref{lem:max}. Let
\[
\bar{w}(t,x):=C_1t\ep^\alpha+C_2\ep
\]
(notice that this function is space-independent) so that 
\[
\frac{d}{dt}\bar w(t,x) - \frac{1}{\sigma_d\ep^{d+2}}\int_{\R^d}\Psi(t,x,y)\left(\bar w(t,y)-\bar w(t,x)\right) K\Big(\frac{y-x}{\ep}\Big)dy=\frac{d}{dt}\bar w(t,x) =C_1\ep^\alpha .
\]
Thanks to \eqref{eq:error} we can choose $C_1>\tilde C$ of \eqref{bound:error} such that $\bar w$ and $w^\ep$ satisfy \eqref{eq:comparison1} and choosing $C_2>0$ we can also get \eqref{eq:comparison2} (since $w^\ep|_{t=0}=0$). Then, \red{since $w^\ep, \bar w$ are both continuous and $2\pi$-periodic},  Lemma \ref{lem:max} yields
\[
u^{I,\ep}(t,x)-u(t,x)=w^\ep(t,x)\leq C_1t\ep^\alpha+C_2\ep
\]
for all $t\in[0,T]$. An analogous reasoning with $-w^\ep$ gives the lower bound and hence proves the result.
\end{proof}

\section{Comparison between microscopic model and integral equation}\label{sec.microtoint}

In this section we prove the convergence of solutions to the discrete model \eqref{eq:km} to solutions of the integral equation \eqref{eq:int}. This, combined with the results of the previous section, gives the proof of Theorem \ref{thm.main}. 

\begin{proposition}\label{ppn:kur-to-int}
Let $n\in\N$ and $V=\{x_1,...,x_n\}$ be a sample of i.i.d.~points with uniform distribution in $\T^d$, $d\ge 1$ and assume $\ep=\ep(n)$ satisfies Condition \ref{assump-eps}. Assume also 
\[
\sum_{n=1}^\infty \P\left(\|u^n_0 - \tilde u_0\|_{L^\infty(V)} > \delta \right)<\infty
\]
for any $\delta>0$. If $u^{n}$ is the unique solution of \eqref{eq:km} with initial condition $u^n_0$, $u^{I,\ep}$ is the unique \red{pseudo-periodic solution of \eqref{eq:int} obtained in Proposition \ref{prop.existence} with initial condition $\tilde u_0\in C^1(\R^d, \R)$ pseudo-periodic,} we have that
\[
\lim_{n\to\infty}\left\|u^{n}-u^{I,\ep}\right\|_{L^\infty([0,T]\times V)}=0, \quad a.s.
\]
\end{proposition} 

\red{As described in Remark \ref{rmk:km-extension}, we extend the initial condition $u_0^n$ from $V$ to $\tilde V$ pseudo-periodically. For the latter, we need to choose and fix the integers $\{\tilde k_\ell\}_{\ell=1}^d$ mentioned in that remark. Since our aim is to approximate the integral equation solution $u^{I,\ep}$ which is pseudo-periodic with winding numbers $\{k_\ell\}_{\ell=1}^d$ (determined by its initial condition $\tilde u_0$), we set  $\tilde k_\ell:=k_\ell$ for $\ell=1,2,...,d$.

We also introduce a convention: For $x_i\in V$ (original point in $\Lambda_d$), its neighbors in the extended cloud are those points $y\in \tilde V$ within Euclidean distance $\ep$ from it. Such $y$ may be outside of $\Lambda_d$, but it is a copy of some $x_j\in V$ inside $\Lambda_d$. Note that this $y$ coincides with the $x^i_j$ defined in Remark \ref{rmk:tor-dist} for $\ep \in (0,1)$. For the rest of this section, for $x_i\in\Lambda_d$, we use $\{x_j,\, j\in\cN(i)\}$ to denote neighbors of $x_i$ in $\tilde V$, and in the case that $x_j$ is outside of $\Lambda_d$, it is not the original point but its copy; we use instead the notation $[x_j]$ (see \eqref{def:mod}) to denote the original point in $\Lambda_d$.}

We first state a discrete maximum principle similar to Lemma \ref{lem:max}.
\begin{lemma}\label{max-disc}
Fix $n\in\N$. Let $a:[0,\infty)\times \tilde V^2\to\R_{\ge0}$ satisfy $a(t,x_i,x_j)>0$ whenever $j\in\cN(i)$, for all $t>0$, and $u, v: [0,\infty)\times \tilde V\to\R$, \red{$2\pi$-periodic in spatial variable in each coordinate direction,} satisfy for any $x_i\in V$
\begin{align*}
\frac{d}{dt}u(t,x_i)- \sum_{j\in\cN(i)}a(t,x_i,x_j)\left(u(t,x_j)-u(t,x_i)\right)  &> \frac{d}{dt}v(t,x_i)-\sum_{j\in\cN(i)}a(t,x_i,x_j)\left(v(t,x_j)-v(t,x_i)\right)   \\
 u(0, x_i)&\geq  v(0, x_i).
\end{align*}
Then, we have that 
\begin{align*}
u(t,x_i)\ge v(t,x_i), \quad \forall t\ge0, \; x_i\in V.
\end{align*}
\end{lemma}

\begin{proof}
We denote $w(t,x_i):=u(t,x_i)-v(t,x_i)$, $x_i\in\tilde V$. The hypothesis can be rewritten as 
\begin{align}\label{contra-disc}
\frac{d}{dt}w(t,x_i)> \sum_{j\in\cN(i)}a(t,x_i,x_j)\left(w(t,x_j)-w(t,x_i)\right), \quad x_i\in V.
\end{align}
Assume the conclusion is false, namely $\inf_{\{t>0,x_i\in V\}}w(t,x_i)<0$. Since $w(0,x_i)\ge0$ for all $x_i\in V$ by hypothesis, there exists a time $t_*$ such that
\[
t_*:=\inf\Big\{t\ge0: \min_{x_i\in V}w(t,x_i)<0\Big\}.
\]
The index $i$ that achieves this first crossing of zero may be non-unique, in which case we take an arbitrary one, call it $i_*$. On the one hand, we must have $\displaystyle{\frac{d}{dt}}w(t_*, x_{i_*})\le 0$, and on the other hand, $w(t_{*}, x_{i_*})=\min_{x_j\in V}w(t_*, x_j)$. With $a(t_*, x_{i_*}, x_j)>0$ for any $j\in\cN(i_*)$ by hypothesis, we have that 
\begin{align*}
&\sum_{j\in\cN(i_*)}a(t_*, x_{i_*},x_j)\left(w(t_*,x_j)-w(t_*,x_{i_*})\right)\\
&=\sum_{j\in\cN(i_*)}a(t_*, x_{i_*},x_j)\left(w(t_*,[x_j])-w(t_*,x_{i_*})\right)\ge 0,
\end{align*}
\red{where we used the $2\pi$-periodicity of $w$ on $\tilde V$.} This is in contradiction with \eqref{contra-disc} at $t=t_*$, $i=i_*$.
\end{proof}


Fix $T$ finite. Let us denote the difference we want to estimate
\begin{align*}
e_n(t,x):=u^{n}(t,x)-u^{I,\ep}(t,x), \quad x\in \tilde V, \, t\ge0.
\end{align*}
\red{Since both $u^n(t,\cdot), u^{I,\ep}(t,\cdot)$ are pseudo-periodic on $\tilde V$ with the same winding numbers $\{k_\ell\}_{\ell=1}^d$, their difference $e_n(t,\cdot)$ is $2\pi$-periodic on $\tilde V$ in each coordinate direction.}
Then, for every $x_i\in V$, and due to \ref{eq:ext-kur}, we have that
\begin{align}
\frac{d}{dt}e_n(t,x_i)&=\frac{1}{\ep^{2}N_i}\sum_{\red{x_j\in\tilde V}}\left[\sin \left(u^{n}(t,x_j)-u^{n}(t,x_i)\right)-\sin \left(u^{I,\ep}(t,x_j)-u^{I,\ep}(t,x_i)\right)\right]K\left(\ep^{-1}(x_j-x_i)\right) \nonumber\\
&\qquad+\Bigg[\frac{1}{\ep^{2}N_i}\sum_{\red{x_j\in\tilde V}}\sin \left(u^{I,\ep}(t,x_j)-u^{I,\ep}(t,x_i)\right)K\left(\ep^{-1}(x_j-x_i)\right) \nonumber\\
&\hspace{3cm} -\frac{1}{\sigma_d\ep^{d+2}}\int_{\red{\R^d}}\sin\left((u^{I,\ep}(t,y)-u^{I,\ep}(t,x_i)\right)K\left(\ep^{-1}(y-x_i)\right)dy\Bigg]  \nonumber\\
&=:A(t,i)+B(t,i).  \label{triangle}
\end{align}
Our main task is to analyse the two difference terms marked as $A(t,i)$ and $B(t,i)$. 

\begin{lemma}\label{lem:tail-bd}
For every $T<\infty$ and \red{$\tilde u_0\in C^1(\R^d,\R)$ pseudo-periodic}, there exist constants $C, C'\in(0,\infty)$ such that for any $\delta>0$, $n\in\N$ and $\ep\in(0,1)$ we have that 
\begin{align*}
\P\Big(\sup_{x_i\in V}\sup_{t\in[0,T]}|B(t,i)|\ge\delta\Big)\le Cn^{3}\big(e^{-C'\ep^{d+2} n \delta^2}+e^{-C'\ep^dn}\big).
\end{align*}
\end{lemma}

\begin{proof}
\red{Recall that we are considering the neighbors on the extended point cloud $\tilde V$.}
Conditional on $x_i\in V$ and the set of indices $\cN(i)$ of neighboring points, all the $x_j$, $j\in\cN(i)$ are i.i.d.~uniformly distributed in $\B(x_i,\ep)$, the Euclidean ball of radius $\ep$ centered around $x_i$. Indeed, knowing that $j\in\cN(i)$ reveals no more information other than $x_j\in \B(x_i,\ep)$. The conditional density of such an $x_j$ equals $|B(x_i,\ep)|^{-1}=(\sigma_d\ep^d)^{-1}$.
For every fixed $t$ and conditional on $x_i$ and $\cN(i)$, the random variables
\begin{align}\label{def-xi}
\xi_j^i(t)&:=\sin\left(u^{I,\ep}(x_j,t)-u^{I,\ep}(x_i,t)\right)K\left(\ep^{-1}(x_j-x_i)\right), \quad j\in\cN(i) 
\end{align}
are i.i.d. By the mean value theorem and Lemma \ref{lem:lip}, their absolute values are a.s. bounded by
\begin{align}\label{xi}
|\xi_j^i(t)|\le |u^{I,\ep}(x_j,t)-u^{I,\ep}(x_i,t)|K\left(\ep^{-1}(x_j-x_i)\right)|\le C_T\ep\|K\|_{L^\infty},
\end{align}
where we used that $|\sin x|\le |x|$ and $K$ has unit-size support. We have that
\begin{align*}
\E\left[\xi_j^i(t)\, |\, x_i, \cN(i)\right] =\frac{1}{\sigma_d\ep^d}\int_{\R^d} \sin\left(u^{I,\ep}(y,t)-u^{I,\ep}(x_i,t)\right)K\left(\ep^{-1}(y-x_i)\right) dy.
\end{align*}
Recall Hoeffding's concentration inequality cf.~\cite[Theorem 2.8]{boucheron}: {\it{Let $Y_1,...,Y_n$ be independent random variables such that $a_i\le Y_i\le b_i$. Let $S_n=\sum_{i=1}^nY_i$ and $\delta>0$, then we have that
\begin{align*}
\P\left(|S_n-\E(S_n)|>\delta\right)\le 2e^{-\frac{2\delta^2}{\sum_{i=1}^n(b_i-a_i)^2}}.
\end{align*}}}
We can apply this inequality to the conditionally independent variables $\xi_j^i(t)$, $j\in\cN_i$ (given $x_i, \cN_i$), and for any fixed $t\in[0,T]$, $\delta>0$, we have that

\begin{align}\label{bound-fixed-t}
&\P\Big(\sup_{x_i\in V}\Big|\frac{1}{\ep^2N_i}\sum_{j\in\cN(i)}\xi_j^i(t)-\frac{1}{\sigma_d \ep^{d+2}}\int_{\R^d}\sin\left((u^{I,\ep}(t,y)-u^{I,\ep}(t,x_i)\right)K\left({\ep}^{-1}{(y-x_i)}\right)dy\Big|\ge \delta\Big)\nonumber\\
&\le\sum_{i=1}^n\P\Big(\Big|\frac{1}{\ep^2N_i}\sum_{j\in\cN(i)}\xi_j^i(t)-\frac{1}{\ep^2}\E\left(\xi_j^i(t)\,|\, x_i, \cN(i)\right)\Big|\ge \delta\Big)\nonumber\\
&=\sum_{i=1}^n\E\Big[\P\Big(\Big|\sum_{j\in\cN(i)}\xi_j^i(t)-N_i\E\left(\xi_j^i(t)\,|\, x_i, \cN(i)\right)\Big|\ge \ep^2N_i\delta\, \Big|\, x_i, \cN(i)\Big)\Big] \nonumber\\
&\le 2\sum_{i=1}^n \E\Big[ e^{-\frac{2(\ep^2N_i\delta)^2}{C_TN_i\ep^2\|K\|_{L^\infty}^2}}\Big] =2\sum_{i=1}^n \E \left[e^{-C'\ep^2 N_i \delta^2}\right] \nonumber\\
&\le Cne^{-C'\ep^{d+2} n \delta^2}+Cne^{-C'\ep^dn},
\end{align}
where in the third step, we used the law of iterated expectation, and the last step is due to our control on $N_i$, namely \eqref{bernstein} applied with $\lambda=\frac{\sigma_d\ep^dn}{2(2\pi)^d}$, and $C,C'$ are finite constants depending only on $C_T,  \|K\|_{L^\infty}$ that in the sequel may change from line to line. Clearly, the regime of $\ep$ in Condition \ref{assump-eps} renders the last expression summable in $n$, for every fixed $\delta>0$.

We want to move a step further and obtain uniformity in $t\in[0,T]$. We first observe that by \eqref{eq:int}, the time derivative of $u^{I,\ep}$ admits a crude bound of $C\ep^{-1}$. Indeed, by Lemma \ref{lem:lip} and $|\sin x|\le |x|$, for any $x\in\Lambda_d$,
\begin{align*}
\left|\displaystyle{\frac{d}{dt}}u^{I,\ep}(t,x)\right| & \le \displaystyle{\frac{1}{\sigma_d\ep^{d+2}}}\int_{\R^d}\left|\sin\left(u^{I,\ep}(t,y)-u^{I,\ep}(t,x)\right)\right|K\left(\ep^{-1}(y-x)\right)dy\\
&\le \displaystyle{\frac{1}{\sigma_d\ep^{d+2}}}\int_{\R^d}\left|u^{I,\ep}(t,y)-u^{I,\ep}(t,x)\right|K\left(\ep^{-1}(y-x)\right)dy\\
&\le  \displaystyle{\frac{1}{\sigma_d\ep^{d+2}}}\int_{\R^d}C_T\left|y-x\right|K\left(\ep^{-1}(y-x)\right)dy\le C_T\kappa_1\ep^{-1},
\end{align*}
where $\kappa_1$ is as defined in \eqref{kappa-2} and a change of variable is used in the last step.
Now we divide the time interval $[0, T]$ into $M:=\lceil T\ep^{-4}\rceil$ sub-intervals of length at most $\ep^4$, i.e.~$0=t_0<t_1<...<t_M=T$ such that $t_{k+1}-t_k\le \ep^4$, $k=0,1,...,M$. For any \red{$x\in\Lambda_d$}, $k$ and $t\in[t_k, t_{k+1}]$, we have that 
\[
|u^{I,\ep}(t,x)-u^{I,\ep}(t_{k},x)|\le C_T\ep^{-1}|t-t_k|\le C\ep^3.
\]
Hence, by the $1$-Lipschitz property of sine function, a.s. for every $t\in[t_k,t_{k+1}]$ we have that 
\begin{align*}
&\sup_{k=1}^M\sup_{x_i\in V}\left|\frac{1}{N_i}\sum_{j\in\cN(i)}\sin\left(u^{I,\ep}(t,x_j)-u^{I,\ep}(t,x_i)\right)-\frac{1}{N_i}\sum_{j\in\cN(i)}\sin\left(u^{I,\ep}(t_{k},x_j)-u^{I,\ep}(t_{k},x_i)\right)\right|\\
&\le \sup_{k=1}^M\sup_{x_i\in V}\sup_{j\in\cN(i)}\left|\sin\left(u^{I,\ep}(t,x_j)-u^{I,\ep}(t,x_i)\right)-\sin\left(u^{I,\ep}(t_k,x_j)-u^{I,\ep}(t_{k},x_i)\right)\right|\\
&\le \red{\sup_{k=1}^M\sup_{x_i\in V}\sup_{j\in\cN(i)}\left|\sin\left(u^{I,\ep}(t,[x_j])-u^{I,\ep}(t,x_i)\right)-\sin\left(u^{I,\ep}(t_k,[x_j])-u^{I,\ep}(t_{k},x_i)\right)\right|}\\
&\le\red{ \sup_{k=1}^M\sup_{x_i\in V}\sup_{j\in\cN(i)}|u^{I,\ep}(t,[x_j])-u^{I,\ep}(t_{k},[x_j])|+
|u^{I,\ep}(t,x_i)-u^{I,\ep}(t_{k},x_i)|}\\
&\le C\ep^3,
\end{align*}
\red{where the second inequality is due to pseudo-peridocity of $u^{I,\ep}(t,\cdot)$ and $2\pi$-periodicity of sine.}
Similarly, a.s. for every $t\in[t_k, t_{k+1}]$ we also have that 
\begin{align*}
&\sup_{k=1}^M\sup_{x_i\in V}\left|\displaystyle{\frac{1}{\sigma_d\ep^d}}\int_{\R^d}\left|\sin\left(u^{I,\ep}(t,y)-u^{I,\ep}(t,x_i)\right)\right|K\left(\ep^{-1}(y-x)\right)dy \right. \\
&\qquad\quad\quad\quad \left.-\displaystyle{\frac{1}{\sigma_d\ep^d}}\int_{\R^d}\left|\sin\left(u^{I,\ep}(t_k,y)-u^{I,\ep}(t_k,x_i)\right)\right|K\left(\ep^{-1}(y-x_i)\right)dy \right|\\
&\le \|K\|_{L^\infty(V)}\sup_{k=1}^M\sup_{x_i\in V}\sup_{y: |y-x_i|\le\ep}\left|\sin\left(u^{I,\ep}(t,y)-u^{I,\ep}(t,x_i)\right)-\sin\left(u^{I,\ep}(t_{k},y)-u^{I,\ep}(t_{k},x_i)\right)\right|\\
&\le \red{\|K\|_{L^\infty(V)}\sup_{k=1}^M\sup_{x_i\in V}\sup_{y: |y-x_i|\le\ep}\left|\sin\left(u^{I,\ep}(t,[y])-u^{I,\ep}(t,x_i)\right)-\sin\left(u^{I,\ep}(t_{k},[y])-u^{I,\ep}(t_{k},x_i)\right)\right|}\\
&\le \red{\|K\|_{L^\infty(V)}\sup_{k=1}^M\sup_{x_i\in V}\sup_{y: |y-x_i|\le\ep}|u^{I,\ep}(t,[y])-u^{I,\ep}(t_k,[y])|+|u^{I,\ep}(t,x_i)-u^{I,\ep}(t_k,x_i)|}\\
&\le C\ep^3.
\end{align*}
Now, for every $t\in [0,T]$, there exists a unique $k\in\{0,1,..,M\}$ such that $t\in[t_k,t_{k+1}]$, and by the triangle inequality and the preceding two estimates
\begin{align*}
&\Big|\frac{1}{\ep^2N_i}\sum_{j\in\cN(i)}\xi_j^i(t)- \frac{1}{\sigma_d\ep^{d+2}}\int_{\red{\mathbb{R}^d}} \sin\left(u^{I,\ep}(t,y)-u^{I,\ep}(t,x_i)\right)K\left(\ep^{-1}(y-x)\right)dy \Big|\\
& \le\Big|\frac{1}{\ep^2N_i}\sum_{j\in\cN(i)}\xi_j^i(t_k) -\frac{1}{\sigma_d\ep^{d+2}}\int_{\red{\mathbb{R}^d}} \sin\left(u^{I,\ep}(t_k,y)-u^{I,\ep}(t_k,x_i)\right)K\left(\ep^{-1}(y-x)\right)dy \Big| +C\ep,
\end{align*}
where we used the shorthand notation \eqref{def-xi} of $\xi_j^i$. 
Hence, we can revise our previous bound \eqref{bound-fixed-t} as follows: given any $\delta>0$ and every $x_i\in V$,
\begin{align*}
&\P\Big(\sup_{t\in[0,T]}\Big|\frac{1}{\ep^2N_i}\sum_{j\in\cN(i)}\xi_j^i(t)-\frac{1}{\sigma_d\ep^{d+2}}\int_{\red{\mathbb{R}^d}} \sin\left(u^{I,\ep}(t,y)-u^{I,\ep}(t,x_i)\right)K\left(\ep^{-1}(y-x)\right)dy \Big|\ge \delta\Big)\\
&\le \P\Big(\exists k_*=1,...,M \text{ s.t. }\Big|\frac{1}{\ep^2N_i}\sum_{j\in\cN(i)}\xi_j^i(t_{k_*})\\
&\quad\quad\quad\quad\quad\quad\quad-\frac{1}{\sigma_d\ep^{d+2}}\int_{\red{\mathbb{R}^d}} \sin\left(u^{I,\ep}(t_{k_*},y)-u^{I,\ep}(t_{k_*},x_i)\right)K\left(\ep^{-1}(y-x)\right)dy \Big|\ge \delta-C\ep\Big)\\
&\le \sum_{k=1}^M\P\Big(\sup_{t\in[0,T]}\Big|\frac{1}{\ep^2N_i}\sum_{j\in\cN(i)}\xi_j^i(t_k)\\
&\quad\quad\quad\quad\quad\quad\quad-\frac{1}{\sigma_d\ep^{d+2}}\int_{\mathbb{R}^d} \sin\left(u^{I,\ep}(t_k,y)-u^{I,\ep}(t_k,x_i)\right)K\left(\ep^{-1}(y-x)\right)dy \Big|\ge \delta-C\ep\Big).
\end{align*}
Note that for $\ep$ small enough, $\delta-C\ep>\delta/2$, and by Condition \ref{assump-eps},
\[
M=\lceil T \ep^{-4}\rceil\le Tn^{\frac{4}{d+2}},
\]
applying our previous bound \eqref{bound-fixed-t} for fixed $t=t_k$, $k=1,2,...,M$ here, we deduce that 
\begin{align*}
&\P\Big(\sup_{x_i\in V}\sup_{t\in[0,T]}\Big|\frac{1}{\ep^2N_i}\sum_{j\in\cN(i)}\xi_j^i(t) -\frac{1}{\sigma_d\ep^{d+2}}\int_{\red{\mathbb{R}^d}} \sin\left(u^{I,\ep}(y,t)-u^{I,\ep}(x_i,t)\right)K\left(\ep^{-1}(y-x)\right)dy \Big|\ge \delta\Big)\\
&\le Cnn^{\frac{4}{d+2}}\big(e^{-C'\ep^{d+2} n \delta^2}+e^{-C'\ep^dn}\big).
\end{align*}
Note that under Condition \ref{assump-eps}, $e^{-C'\ep^{d+2} n \delta^2}$ and $e^{-C'\ep^dn}$ decay faster than any polynomial in $n$, hence the preceding display is still summable in $n$, for every fixed $\delta$. This completes the proof of Lemma \ref{lem:tail-bd}.
\end{proof}

\begin{proof}[Proof of Proposition \ref{ppn:kur-to-int}]
Recall the term $A(t,i)$ in \eqref{triangle} and that $e_n(t,x):=u^{n}(t,x)-u^{I,\ep}(t,x)$ for $x\in\tilde V$. 
By the mean value theorem, for any $x_j\in\tilde V$ neighbor of $x_i\in V$, we have that 
\begin{align*}
&\sin \left(u^{n}(t,x_j)-u^{n}(t,x_i)\right)-\sin \left(u^{I,\ep}(t,x_j)-u^{I,\ep}(t,x_i)\right)\\
&=\cos \left[\beta\left(u^{n}(t,x_j)-u^{n}(t,x_i)\right)+(1-\beta)\left(u^{I,\ep}(t,x_j)-u^{I,\ep}(t,x_i)\right)\right]\left(e_n(t,x_j)-e_n(t,x_i)\right)
\end{align*}
for some $\beta\in(0,1)$. Let us denote
\begin{align*}
\Xi(x_i,x_j, t)&:=\cos\left(\beta\left(u^{n}(t,x_j)-u^{n}(t,x_i)\right)+(1-\beta)\left(u^{I,\ep}(t,x_j)-u^{I,\ep}(t,x_i)\right)\right)\\
&=\cos\left(\beta\left(e_n(t,x_j)-e_n(t,x_i)\right)+\left(u^{I,\ep}(t,x_j)-u^{I,\ep}(t,x_i)\right)\right),
\end{align*}
where the second line is just a rearrangement of terms.
Define the random time
\begin{align*}
\tau_n:=\inf\Big\{t\ge0: \sup_{x_i\in V}\left|e_n\left(t,x_i\right)\right|\ge\frac{\pi}{16}\Big\}.
\end{align*}
By Lemma \ref{lem:lip}, for such $x_i,x_j$ as above, we have $|u^{I,\ep}(t,x_j)-u^{I,\ep}(t,x_i)|\le C_T|x_j-x_i|\le C_T\ep$. Taking $\ep$ small enough, 
for any $\beta\in(0,1)$ we have that 
\begin{align*}
&\sup_{x_i\in V}\sup_{j\in\cN(i)}\sup_{t\in[0,\tau_n]}\left|\beta\left(e_n(t,x_j)-e_n(t,x_i)\right)+\left(u^{I,\ep}(t,x_j)-u^{I,\ep}(t,x_i)\right)\right|\le2\times\frac{\pi}{16}+C_T\ep\le \frac{\pi}{4},
\end{align*}
and hence we have
\begin{align}\label{pos-coeff}
&\inf_{x_i\in V}\inf_{j\in\cN(i)}\inf_{t\in[0,\tau_n]}\Xi(x_i,x_j, t)\ge \frac{1}{\sqrt{2}}.
\end{align}
The whole expression in \eqref{triangle} for $t\in[0,T]$ can be written and bounded as (where $x_i\in V$)
\begin{align}\label{div-form}
&\frac{d}{dt}e_n(t,x_i)=A(t,i)+B(t,i)\nonumber\\
&< \frac{1}{\ep^{2}N_i}\sum_{x_j\in\tilde V}\Xi(x_i,x_j, t)K\left(\ep^{-1}(x_j-x_i)\right)\left(e_n(t,x_j)-e_n(t,x_i)\right)+\sup_{x_i\in V}\sup_{s\in[0,T]} |B(s,i)|+\ep,
\end{align}
and we have added an $\ep$ to maintain a strict inequality. We want to apply the maximum principle Lemma \ref{max-disc}.
Denote 
\[
\zeta_{n}(t):=t\sup_{x_i\in V}\sup_{s\in[0,T]} |B(s,i)| +\ep t+ \|u^n_0 -\tilde  u_0\|_{L^\infty(V)},
\]
where $u^n_0$ is the initial condition for the Kuramoto equation, and $\tilde u_0$ the initial condition for the integral equation. \red{We already remarked that both $e_n, \zeta_n$ are $2\pi$-periodic in spatial variable.}


By Lemma \ref{max-disc} applied to \eqref{div-form}, taking there 
\[
u(t,x)=\zeta_n(t), \quad v(t,x)=e_n(t,x), \quad x\in \tilde V,
\]
with $u(0)\ge v(0)$ holding, we conclude that for any $t\in[0,\tau_n\wedge T]$,
\begin{align*}
\sup_{x_i\in V}e_n(t,x_i)\le  \zeta_{n}(T).
\end{align*}
Analogous arguments applied to $-e_n(t,x)$, $t\in[0,\tau_n\wedge T]$ yields $\sup_{x_i\in V}-e_n(t,x_i)\le  \zeta_{n}(T)
$, hence taken together
\begin{align}\label{disc-max}
\sup_{x_i\in V}|e_n(t,x_i)|\le \zeta_{n}(T), \quad t\in[0,\tau_n\wedge T].
\end{align}
By Lemma \ref{lem:tail-bd} and \eqref{disc-max}, we have that for any $\delta>0$,
\begin{align}\label{tail-en}
\P\Big(\sup_{t\in[0,T\wedge \tau_n]}\sup_{x_i\in V}| & e_n(t,x_i)|>\delta\Big) \nonumber\\
& \le \P\Big(T\sup_{x_i\in V}\sup_{t\in[0,T]} |B(t,i)| >\delta/2\Big) + \P\left(\|u^n_0 - u_0\|_{L^\infty(V)} +\ep T > \delta/2 \right)\nonumber\\
& \le Cn^{3}\left(e^{-C'\ep^{d+2} n \delta^2}+e^{-C'\ep^dn}\right) + \P\left(\|u^n_0 - u_0\|_{L^\infty(V)} > \delta/4 \right),
\end{align}
upon considering $\ep$ small enough. Further, considering only $\delta<\pi/16$, we have that
\begin{align*}
\P(\tau_n\le T)\le \P\Big(\sup_{t\in[0,T\wedge \tau_n]}\sup_{x_i\in V}|e_n(t,x_i)|\ge\pi/16\Big)\le \P\Big(\sup_{t\in[0,T\wedge \tau_n]}\sup_{x_i\in V}|e_n(t,x_i)|>\delta\Big).
\end{align*}
Hence by \eqref{tail-en}, for such $\delta$, 
\begin{align}\label{summable}
&\P\Big(\sup_{t\in[0,T]}\sup_{x_i\in V}|e_n(t,x_i)|>\delta\Big)\nonumber\\
&=\P\Big(\sup_{t\in[0,T\wedge\tau_n]}\sup_{x_i\in V}|e_n(t,x_i)|>\delta;\, T<\tau_n\Big)+\P\Big(\sup_{t\in[0,T]}\sup_{x_i\in V}|e_n(t,x_i)|>\delta;\, \tau_n\le T\Big)\nonumber\\
&\le \P\Big(\sup_{t\in[0,T\wedge \tau_n]}\sup_{x_i\in V}|e_n(t,x_i)|>\delta\Big)+\P(\tau_n\le T)\nonumber\\
&\le Cn^{3}\left(e^{-C'\ep^{d+2} n \delta^2}+e^{-C'\ep^dn}\right) + 2\P\left(\|u^n_0 - u_0\|_{L^\infty(V)} > \delta/4 \right).
\end{align}
Since the last expression is summable in $n$ for every fixed $\delta$, by the Borel--Cantelli lemma, 
\[
\P\Big(\sup_{t\in[0,T]}\sup_{x_i\in V}|e_n(t,x_i)|>\delta, \, \text{i.o.}\Big)=0,
\]
where i.o. means infinitely often in $n$. 
Since the events $\mathsf A_\delta:=\big\{\sup_{t\in[0,T]}\sup_{x_i\in V}|e_n(t,x_i)|>\delta, \,  \text{i.o.}\big\}$ are nested in $\delta$ and non-decreasing as $\delta\downarrow 0$, we conclude that 
\[
\P\Big(\limsup_{n\to\infty}\sup_{t\in[0,T]}\sup_{x_i\in V}|e_n(t,x_i)|>0\Big) \le \P(\cup_{\delta>0}\mathsf A_\delta)
= \lim_{\delta\downarrow 0}\P(\mathsf A_\delta)=0.
\]
This completes the proof of Proposition \ref{ppn:kur-to-int} and Theorem \ref{thm.main}.
\end{proof}

\section{Simulations}
\label{sec.simulations}

We illustrate our results with some simulations with different initial conditions. We sample $n=2000$ independent uniform points in $\T^2$ and we construct the random geometric graph with $\ep_n=.25$. In Figure \ref{planar} we show the results representing the points in  $[0,2\pi]^2$ and in Figure \ref{3d} we show them embedded in $\R^3$. In both cases we consider solutions of \eqref{eq:km} with initial conditions $u_0$ as displayed in the leftmost column of both Figure \ref{planar} and Figure \ref{3d}. Observe that all of them represent stable equilibria for the heat equation \eqref{eq:heat}. In view of this fact and  Theorem \ref{thm.main} we expect the solutions of the Kuramoto model \eqref{eq:km} to remain close to these initial conditions at least in finite time intervals (and arguably for all times). In each figure, each row represents a different initial condition. Different columns represent different moments in time. The first column always represents $t=0$. The second and third columns show snapshots at moderate time $t=5,11,14$. We ran the simulations up to time 100 and no change can be appreciated. In all the cases the simulations indicate that the situation is stable at least in this range of times. To improve visualization and to emphasize the twisted states we show $u^n(\cdot, t) + t$ instead of $u^n(\cdot, t)$ (which corresponds to $\omega_i=1$ for all $i$ in \eqref{kuramoto.graph}). Figure $\ref{palette}$ shows the color representation of each phase $\theta \in \mathbb S^1$. Videos of these solutions in motion and the code used to generate them can be found in \url{https://github.com/FranCire/TorusKuramoto}. \pagebreak


\vspace*{\fill} 
\begin{center}
\begin{minipage}{0.8\linewidth}
  \centering
  \small
  \begin{tabular}{p{0.15\linewidth}|ccc}
    Initial\\condition & & Planar Representation &\\
    \hline
    $u_0 = x\cdot e_1$ &
    \parbox[c]{0.24\linewidth}{\includegraphics[width=\linewidth]{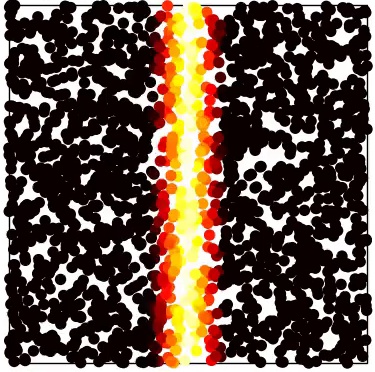}} &
    \parbox[c]{0.24\linewidth}{\includegraphics[width=\linewidth]{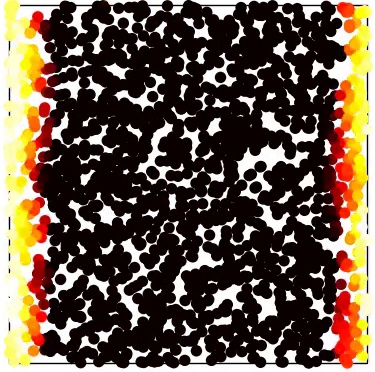}} &
    \parbox[c]{0.24\linewidth}{\includegraphics[width=\linewidth]{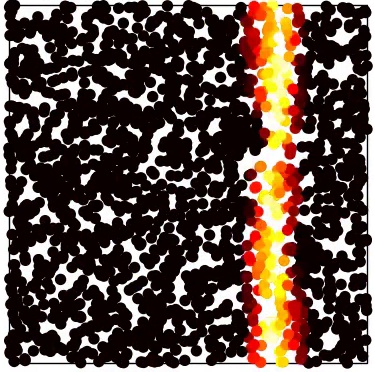}} \\
    & $t=0$ & $t=5$ & $t=11$ \\
    \hline
    $u_0 = x\cdot e_2$ &
    \parbox[c]{0.24\linewidth}{\includegraphics[width=\linewidth]{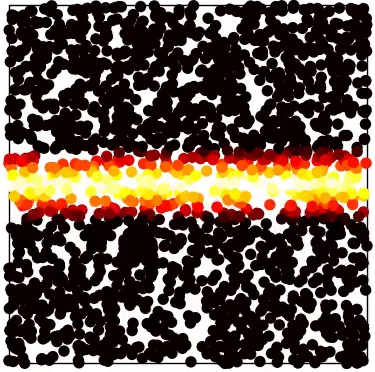}} &
    \parbox[c]{0.24\linewidth}{\includegraphics[width=\linewidth]{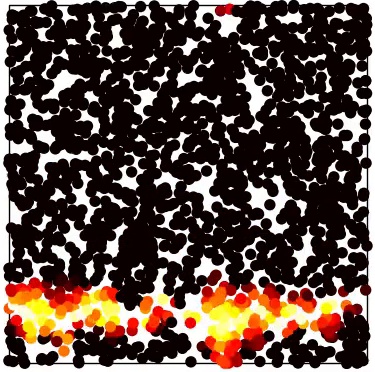}} &
    \parbox[c]{0.24\linewidth}{\includegraphics[width=\linewidth]{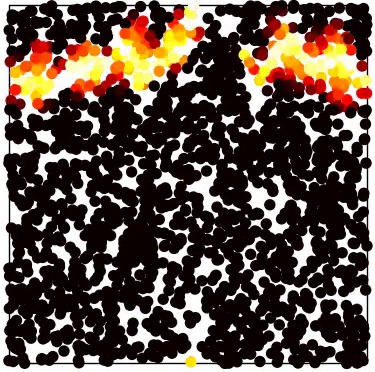}} \\
    & $t = 0$ & $t = 2$ & $t = 11$ \\
    \hline
    $u_0=2x\cdot e_1$ &
    \parbox[c]{0.24\linewidth}{\includegraphics[width=\linewidth]{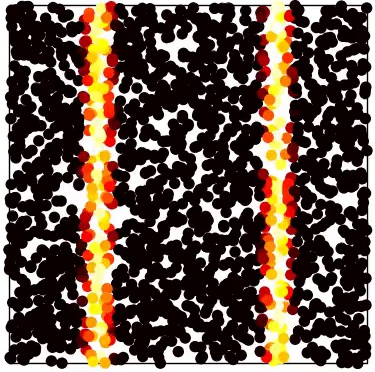}} &
    \parbox[c]{0.24\linewidth}{\includegraphics[width=\linewidth]{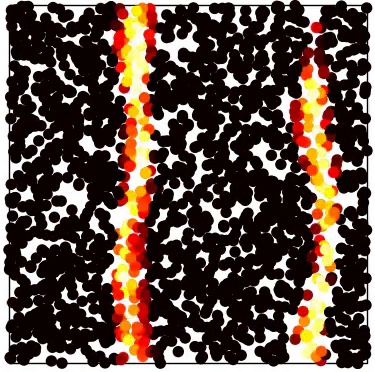}} &
    \parbox[c]{0.24\linewidth}{\includegraphics[width=\linewidth]{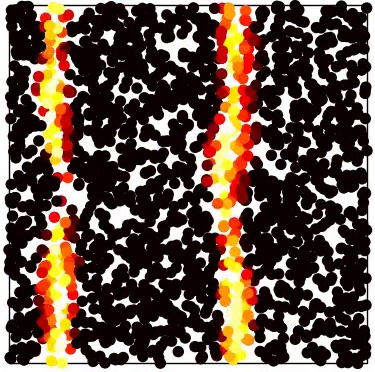}} \\
    & $t=0$ & $t=5$ & $t=14$ \\
    \hline
    $u_0=x\cdot(1,1)$ &
    \parbox[c]{0.24\linewidth}{\includegraphics[width=\linewidth]{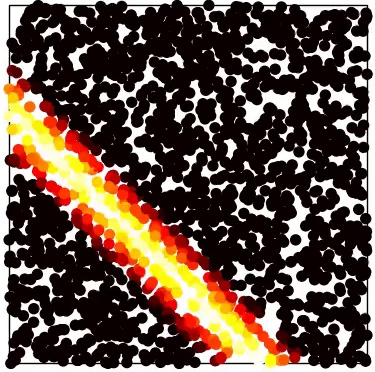}} &
    \parbox[c]{0.24\linewidth}{\includegraphics[width=\linewidth]{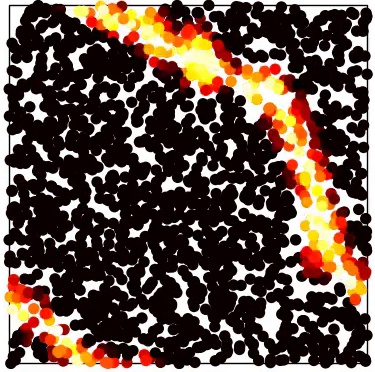}} &
    \parbox[c]{0.24\linewidth}{\includegraphics[width=\linewidth]{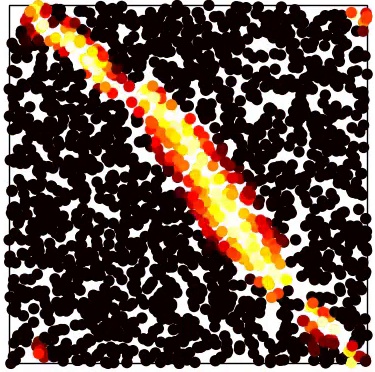}} \\
    & $t=0$ & $t=5$ & $t=11$ \\
  \end{tabular}
  \captionof{figure}{Planar Representation}
  \label{planar}
\end{minipage}
\end{center}
\vspace*{\fill} 

\newpage
  \vspace*{\fill} 
\begin{center}
  \begin{minipage}{0.8\linewidth}
  \centering
  \small
  \begin{tabular}{p{0.15\linewidth}|ccc}
  Initial\\condition & & 3D Representation &\\
  \hline
  $u_0 = x\cdot e_1$ &
  \parbox[c]{0.24\linewidth}{\includegraphics[width=\linewidth]{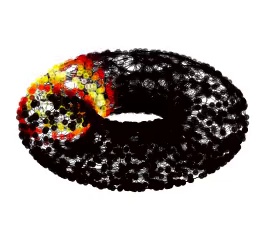}} &
  \parbox[c]{0.24\linewidth}{\includegraphics[width=\linewidth]{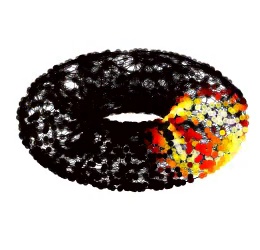}} &
  \parbox[c]{0.24\linewidth}{\includegraphics[width=\linewidth]{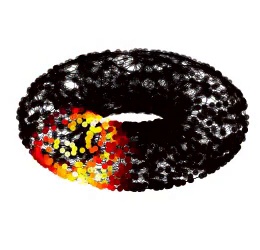}} \\
  & $t=0$ & $t=5$ & $t=11$ \\
  \hline\\
  $u_0 = x\cdot e_2$ &
   \parbox[c]{0.24\linewidth}{\includegraphics[width=\linewidth]{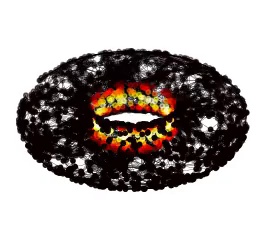}} &
  \parbox[c]{0.24\linewidth}{\includegraphics[width=\linewidth]{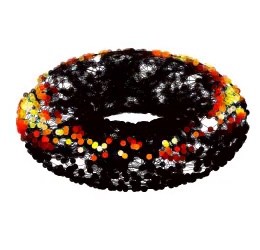}} &
  \parbox[c]{0.24\linewidth}{\includegraphics[width=\linewidth]{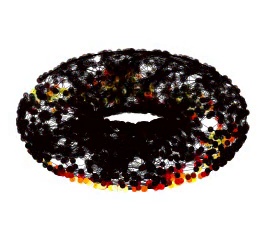}} \\
  & $t = 0$ & $t = 2$ & $t = 11$ \\
  \hline\\
  $u_0=2x\cdot e_1$ &
  \parbox[c]{0.24\linewidth}{\includegraphics[width=\linewidth]{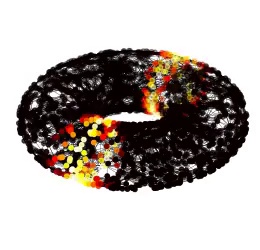}} &
  \parbox[c]{0.24\linewidth}{\includegraphics[width=\linewidth]{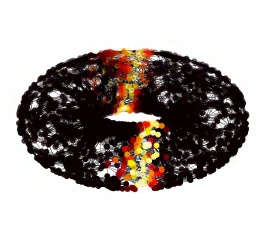}} &
  \parbox[c]{0.24\linewidth}{\includegraphics[width=\linewidth]{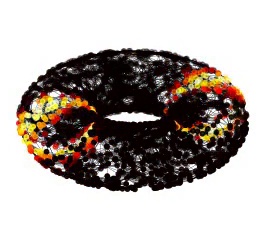}} \\
  & $t=0$ & $t=5$ & $t=14$ \\
  \hline\\
  $u_0=x\cdot(1,1)$ &
  \parbox[c]{0.24\linewidth}{\includegraphics[width=\linewidth]{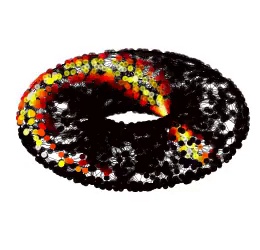}} &
  \parbox[c]{0.24\linewidth}{\includegraphics[width=\linewidth]{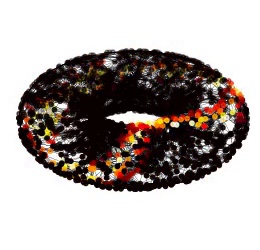}} &
  \parbox[c]{0.24\linewidth}{\includegraphics[width=\linewidth]{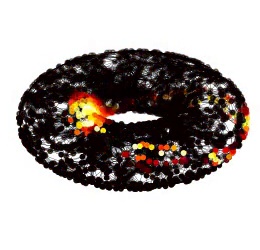}} \\
  & $t=0$ & $t=5$ & $t=11$ \\
  \end{tabular}
  \captionof{figure}{3D Representation}
  \label{3d}
  \end{minipage}
\begin{figure}
    \centering
    \includegraphics[width=.25\linewidth]{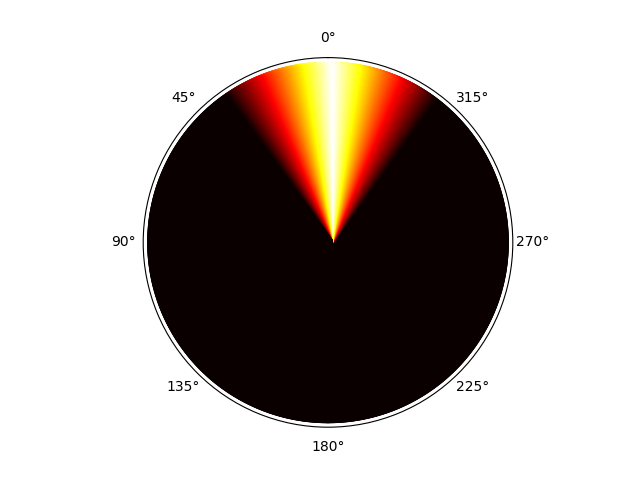}
    \caption{Color Palette}
    \label{palette}
\end{figure}
\end{center}

\vspace*{\fill} 

%
%
%
%
%

\newpage
{\textbf{Acknowledgments.}} 
\red{We thank the anonymous referees whose comments and criticisms have significantly improved our paper.}

PG and HV are partially supported by CONICET Grant PIP 2021 11220200102825CO, UBACyT Grant 20020190100293BA and PICT 2021-00113 from Agencia I+D. RH is funded by  the Deutsche Forschungsgemeinschaft (DFG, German Research Foundation) under Germany's Excellence Strategy EXC 2044-390685587, Mathematics M\"unster: Dynamics-Geometry-Structure.


\bibliographystyle{abbrv}
\bibliography{kuramotoSL}
\end{document}